\documentclass[11pt,letterpaper]{amsart}
\usepackage{amsmath,amssymb}
\usepackage{graphicx}

\usepackage[colorinlistoftodos]{todonotes}
\newcommand{\ve}{\varepsilon}
\newcommand{\supp}{\operatorname{supp}}
\newcommand{\ass}{\quad\mbox{as}\quad}

\newtheorem{theorem}{Theorem}[section]
\newtheorem{proposition}[theorem]{Proposition}
\newtheorem{lemma}[theorem]{Lemma}

\numberwithin{equation}{section}

\begin{document}
	\title{A Supercritical Problem in Dimension Two}
	
	\author[M. del Pino]{Manuel del Pino}
	\address{Department of Mathematical Sciences, University of Bath, Bath BA2 7AY, United Kingdom.}
	\email{m.delpino@bath.ac.uk}
	
	\author[I. Guerra]{Ignacio Guerra}
	\address{Facultad de Ciencia, Departamento de Matem\'atica y Ciencia de la Computaci\'on, Universidad de Santiago de Chile (USACH), Las Sophoras 173, Estaci\'on Central, 9170020 Santiago, Chile.}
	\email{ignacio.guerra@usach.cl}
	
	\author[M. Musso]{Monica Musso}
	\address{Department of Mathematical Sciences, University of Bath, Bath BA2 7AY, United Kingdom.}
	\email{mm2683@bath.ac.uk}

	\begin{abstract}
		Let $\Omega\subset\mathbb R^2$ be a smooth bounded domain containing the
		origin and invariant under reflection across the coordinate axes, and let
		$0<\lambda<\lambda_1(\Omega)$, where $\Lambda_1 (\Omega)$ is the first eigenvalue for $-\Delta$ on $\Omega$ under Dirichlet boundary conditions.  For every fixed integer $k\ge1$ and all
		sufficiently small $\varepsilon>0$, we construct a positive solution of
		\[
		-\Delta u=\lambda u e^{u^{2+\varepsilon}} \quad\hbox{in }\Omega,
		\qquad
		u=0\quad\hbox{on }\partial\Omega,
		\]
		which blows up at the origin as a tower of $k$ Liouville bubbles.  The
		concentration scales are strongly separated.  The proof is based on a
		Lyapunov--Schmidt reduction adapted to these different scales.
	\end{abstract}
	\maketitle

	\section{Introduction}
	Let $\Omega\subset\mathbb R^2$ be a smooth bounded domain containing the
	origin and invariant under the two coordinate reflections. We consider
	\begin{equation}\label{main}
		\begin{cases}
			-\Delta u=\lambda u e^{u^p} & \text{in }\Omega,\\
			u>0 & \text{in }\Omega,\\
			u=0 & \text{on }\partial\Omega,
		\end{cases}
		\qquad p=2+\ve,
	\end{equation}
	where $0<\lambda<\lambda_1(\Omega)$, with $\lambda_1(\Omega)$ the first
	Dirichlet eigenvalue of $-\Delta$, and $\ve>0$ is small. We seek
	solutions with the symmetries
	\begin{equation}\label{simm}
		u(x_1,x_2)=u(-x_1,x_2)=u(x_1,-x_2).
	\end{equation}
	
	The exponent $p=2$ is critical in dimension two because of the
	Moser--Trudinger inequality \cite{moser,trudinger}.  Existence,
	compactness and blow-up for equations with critical exponential growth
	have been studied extensively; see, for instance,
	\cite{adi,adi2,adi3,adidruet,adistruwe,ap,dmf,druet}.  We are interested
	here in the slightly supercritical problem $p=2+\ve$ and in solutions
	whose blow-up at the origin occurs at several different scales.
	
	Bubble towers for slightly supercritical power nonlinearities were
	constructed in \cite{ddmbn}.  Related multi-bubbling phenomena for
	exponential nonlinearities, in radial or quasilinear settings, were
	studied in \cite{ddmphase,ddmexp}.  More recently, Naimen
	\cite{naimen1,naimen2} studied the supercritical exponential problem in a
	disc by a different radial scaling approach, obtaining detailed
	concentration and oscillation estimates and consequences for the
	bifurcation diagram and multiplicity.  The result below concerns the
	semilinear problem \eqref{main} on a general planar domain satisfying the
	two reflection symmetries.  The basic local model is the Liouville bubble,
	as in \cite{bp,dkm,egp}.
	
	We use the Green function $G_\lambda$ of $-\Delta-\lambda$, normalized by
	\[
	-\Delta_xG_\lambda(x,y)-\lambda G_\lambda(x,y)=8\pi\delta_y
	\quad\hbox{in }\Omega,
	\qquad
	G_\lambda(x,y)=0
	\quad\hbox{on }\partial\Omega.
	\]
	Its regular part and Robin function are
	\begin{equation*}\label{acca}
		H_\lambda(x,y)=4\log\frac1{|x-y|}-G_\lambda(x,y),
	\end{equation*}
	\begin{equation}\label{defrobin}
		h_\lambda(x)=H_\lambda(x,x).
	\end{equation}
	
	Our main result is the following.

	\begin{theorem}\label{teo}
		Assume that $\Omega\subset\mathbb R^2$ is smooth, bounded, contains the
		origin, and is invariant under the two coordinate reflections, see \eqref{simm}. Let
		$0<\lambda<\lambda_1(\Omega)$ and let $k\ge1$ be fixed. Then, for all
		$\ve>0$ sufficiently small, problem \eqref{main} has a positive solution
		$u_\ve$ with a tower of $k$ bubbles concentrating at the origin.
		
		More precisely, there exist scales
		\[
		d_k\ll d_{k-1}\ll\cdots\ll d_1\longrightarrow0 \quad 
		\]
		and heights 
		\[
		\gamma_j
		=
		\alpha_j|\log\ve|^{1/2}\ve^{1/2-j}(1+o(1)),
		\qquad
		\alpha_j=
		\frac{2^{j-1}(k-j)!}{(k-1)!}\sqrt{\frac2k}
		, \quad j=1, \ldots , k\]
		as $\ve \to 0$, such that, for every fixed $j$ and every
		compact set $K\Subset\mathbb R^2\setminus\{0\}$,
		\[
		p\gamma_j^{p-1}
		\bigl[u_\ve(d_jy)-\widehat\gamma_j\bigr]
		\longrightarrow
		w(y):=\log\frac8{(1+|y|^2)^2}
		\]
		uniformly for $y\in K$. Here $\widehat\gamma_j$ are the effective local levels  and satisfy 
		$\widehat\gamma_j/\gamma_j\to1$ as $\ve \to 0$.
		
		The concentration scales satisfy, as $\ve \to 0$,
		\[
		\log d_j^2
		=
		-\gamma_j^p+O\!\left(\ve\gamma_j^p+\log\gamma_j\right),
		\]
		and in particular
		\[
		\frac{d_j}{d_{j-1}}\longrightarrow0,
		\qquad j=2,\ldots,k.
		\]
	\end{theorem}
	
	The idea behind the construction of the solutions found in Theorem \ref{teo} is based on the following. 
	
	For positive parameters $\gamma$ and $\rho$ satisfying
	\[
	\rho^{-2}=\lambda p\gamma^p e^{\gamma^p},
	\]
	with the change of variables
	\[
	u(x) \ = \ \gamma + \frac 1{p \gamma^{p-1}} v \left (
	\frac{x}{\rho}\right ) ,
	\]
	problem \eqref{main} becomes
	\[
	\Delta v + \left(1+\frac{v}{p\gamma^p}\right)
	e^v\exp\left\{\gamma^p\left[
	\left(1+\frac{v}{p\gamma^p}\right)^p
	-1-\frac{v}{\gamma^p}
	\right]\right\}=0
	\quad\hbox{in }\Omega/\rho,
	\]
	with $v>-p\gamma^p$ in $\Omega/\rho$ and $v=-p\gamma^p$ on its boundary.
	Formally, if 
	\[
	\gamma\to\infty,\qquad \rho\to 0,
	\qquad \mbox{as}\quad \ve\to 0, 
	\]
	this equation tends to the Liouville equation
	\begin{equation}
		\label{liou}
		\Delta v + e^v = 0 \quad\mbox{in }\mathbb R^2.
	\end{equation}
	This suggests using Liouville profiles as the building blocks of the approximation. The centered finite-mass radial solutions of \eqref{liou} are
	\begin{equation*}
		\label{defw}
		\omega_{\mu } (y) = \log\frac{8\mu^2}{(\mu^2 + |y|^2)^2},
		\qquad \mu>0.
	\end{equation*}
	We combine $k$ such profiles.  The parameters $\gamma_j, \mu_j$, $j=1,\ldots, k$ asociated to each profile are chosen so that the constant terms match from one scale to the next.
	
	There are two points in the proof where the presence of several widely
	separated scales matters.  The first is the choice of the parameters
	$\mu_j$  entering the physical scales
	\[
	d_j=\mu_j\rho_j.
	\]
	where $\rho_j^{-2}=\lambda p\gamma_j^p e^{\gamma_j^p}$. 
	For the inner bubbles, an error which is small relative to
	$\gamma_j^p$ may still produce a large multiplicative error after
	exponentiation.  We therefore determine the constant part of each bubble
	by an exact scalar matching equation, and only afterwards expand the
	non-constant terms.
	
	The second point is the finite-dimensional system.  The  lineal projected problem contains coefficients $c_i$, $i=1,\ldots,k$ associated to each profile, to be determine to solve the problem. We test the error of a suitable ansatz with an element of the kernel of the linear problem.  The coefficients
	$c_2,\ldots,c_k$ give $k-1$ equations of order $\ve$ and determine the
	relative heights of the bubbles. The remaining parameter is seen only at
	the smaller scale
	\[
	G_1^{-1}\asymp\frac{\ve}{|\log\ve|},
	\qquad G_1=\gamma_1^p.
	\]
	At order $\ve$ the corresponding terms cancel after summing the
	projection equations.  This produces the scalar equation which determines
	the correct scale for $\gamma_1$, see \eqref{def1}. The same separation of scales also enters the linear theory:
	between two consecutive cores we use the logarithmic variable
	$s=-\log|x|$ and a positive bridge supersolution whose size interpolates
	between $\gamma_j^{1-p}$ and $\gamma_{j+1}^{1-p}$.
	
	The numerical diagrams in Figures~\ref{fig1}--\ref{fig2} illustrate the
	multiple turning pattern which motivated the construction. They correspond
	to the problem in a ball and are not used in the proof. The plotted points
	are the same numerical data as in the original computations; only the
	graphical presentation has been redrawn.
	
	\begin{figure}[ht]
		\centering
		\includegraphics[width=0.58\textwidth]{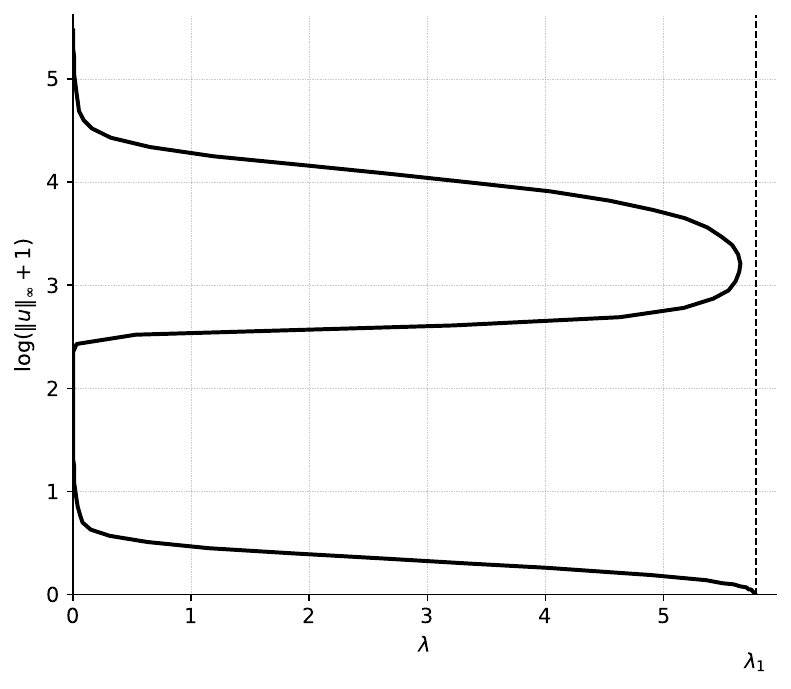}
		\caption{Numerical bifurcation diagram in a ball for $p=2.0002$: 
			$\log(\|u\|_\infty+1)$ versus $\lambda$. The dashed vertical line marks
			$\lambda_1$.}
		\label{fig1}
	\end{figure}
	
	\begin{figure}[ht]
		\centering
		\includegraphics[width=0.92\textwidth]{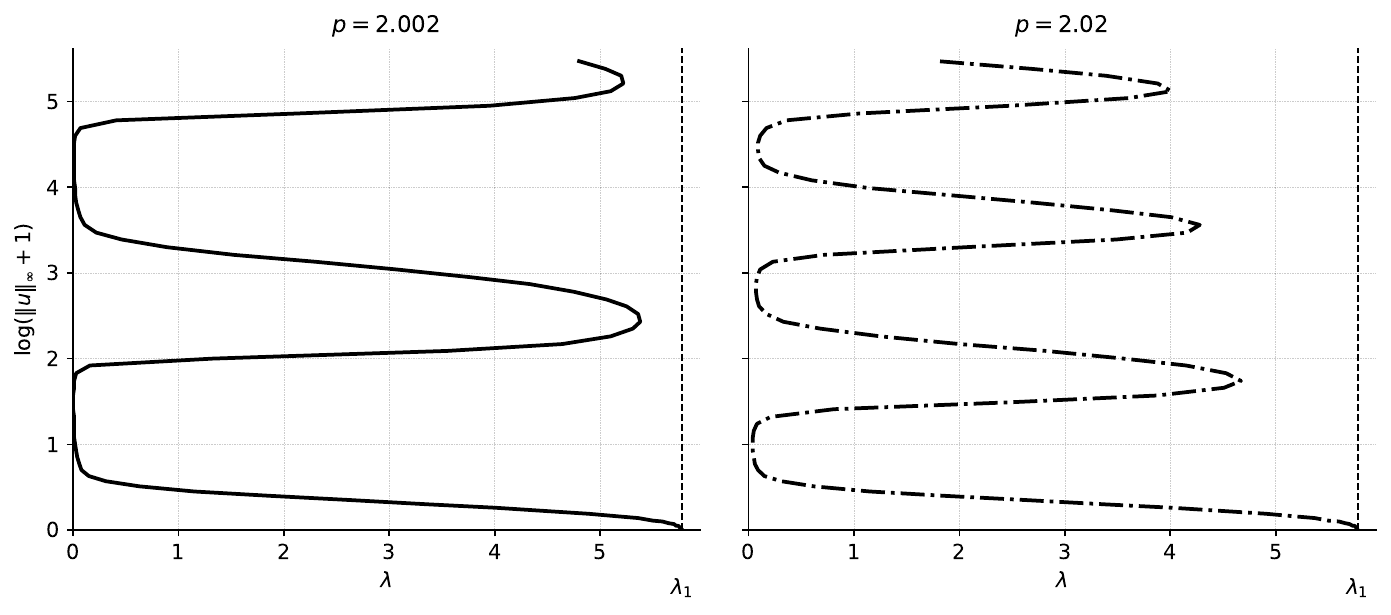}
		\caption{Numerical bifurcation diagrams in a ball. Left: $p=2.002$.
			Right: $p=2.02$. In both panels the dashed vertical line marks
			$\lambda_1$.}
		\label{fig2}
	\end{figure}

	The rest of the paper is organized as follows. Section~\ref{sec2}
	constructs the approximation, introduces the positive weighted norms, and
	estimates the error. Section~\ref{seclinear} proves the uniform projected
	linear theory. Section~\ref{sec:directprojection} computes the fast and slow
	projection equations and closes the reduction.
	
	\section{Construction of the approximation and error estimates}\label{sec2}
	
	We first construct the approximate solution $U$.  Its error is
	\begin{equation}\label{error}
		E:=\Delta U+\lambda Ue^{U^p}
		\quad\hbox{in }\Omega.
	\end{equation}
	The estimates will be made in weighted $L^\infty$ norms adapted to the
	different bubble scales.

	\subsection{Parameters and projected bubbles}
	
	Fix once and for all a small number $\delta\in(0,1)$ and let
	\begin{equation}\label{emme}
		m=(m_1,\ldots,m_k)\in[\delta,\delta^{-1}]^k.
	\end{equation}
	The first height $\gamma_1$ and preliminary scale $\rho_1$ are defined by
	\begin{equation}\label{def1}
		\rho_1^{-2}:=\lambda p\gamma_1^p e^{\gamma_1^p},
		\qquad
		\gamma_1^p e^{-\ve k\gamma_1^p/2}=m_1.
	\end{equation}
	Observe that the second equation has two solutions $\gamma$: with $\gamma_1$ we denote the large solution. Setting
	\[
	G_1:=\gamma_1^p,
	\qquad
	x_\ve:=\frac{k\ve G_1}{2},
	\]
	then
	\[
	x_\ve e^{-x_\ve}=\frac{k\ve m_1}{2},
	\]
	and the large branch is characterized by $x_\ve\to\infty$.  Taking
	logarithms gives
	\[
	x_\ve-\log x_\ve
	=
	|\log\ve|+\log\frac{2}{km_1}.
	\]
	It follows that
	\[
	x_\ve
	=
	|\log\ve|+\log|\log\ve|+O(1), \quad \ass \ve \to 0
	\]
	uniformly for $m_1\in[\delta,\delta^{-1}]$.
	Consequently,
	\begin{equation*}\label{G1size}
		G_1
		=
		\frac{2}{k\ve}
		\left(
		|\log\ve|+\log|\log\ve|+O(1)
		\right),
		\qquad
		G_1^{-1}\approx \frac{\ve}{|\log\ve|} \quad \ass \ve \to 0.
	\end{equation*}
	For $j=2,\ldots,k$, define
	\begin{equation}\label{defj}
		\rho_j^{-2}:=\lambda p\gamma_j^p e^{\gamma_j^p},
		\qquad
		\ve\frac{\gamma_j}{\gamma_{j-1}}=\frac1{m_j}.
	\end{equation}
	Then
	\[
	\gamma_j\to\infty,\qquad \rho_j\to0,
	\qquad j=1,\ldots,k,
	\]
	and
	\begin{equation}\label{explicitgamma}
		\gamma_j
		=
		\frac{\gamma_1}{\prod_{i=2}^j m_i}\,\ve^{-(j-1)},
		\qquad j=2,\ldots,k.
	\end{equation}
	
	Let $\mu_1,\ldots,\mu_k$ be positive parameters, to be fixed below by
	the exact matching equations, and set
	\begin{equation*}\label{djdef}
		d_j:=\mu_j\rho_j,\qquad j=1,\ldots,k.
	\end{equation*}
	We also set
	\begin{equation}\label{bjdef}
		b_j:=\gamma_j^{1-p},
		\qquad
		w_j(x):=w(x/d_j),
		\qquad
		V_j(x):=d_j^{-2}e^{w_j(x)}
		=\frac{8d_j^2}{(d_j^2+|x|^2)^2}.
	\end{equation}
	Lemma~\ref{lem:robust-scales} will show that
	$\log\mu_j=o(\gamma_j^p)$ and that the physical scales $d_j$ are strongly
	separated.
	For $j=1,\ldots,k$, set
	\begin{equation}\label{defWj}
		W_j(x)
		=
		\bar\gamma_j
		+\frac{1}{p\gamma_j^{p-1}}
		\omega_{\mu_j}(x/\rho_j),
	\end{equation}
	where the adjusted height $\bar\gamma_j$ is defined by
	\begin{equation*}
		\label{defbargammaj}
		\bar \gamma_j = \gamma_j + \eta_j,
	\end{equation*}
	where
	\begin{equation}
		\label{defetaj}
		\eta_1=0,\qquad
		\eta_j=-\sum_{i<j}\frac{-4\log d_i-h_\lambda(0)}{p\gamma_i^{p-1}} \quad j=2, \ldots , k,
	\end{equation}
	where $h_\lambda$ is the Robin function as defined in \eqref{defrobin}.
	After the matching parameters have been fixed, Lemma~\ref{lem:robust-scales}
	implies
	\[
	\eta_j
	=
	-\frac2p\gamma_{j-1}(1+o(1))
	=
	-\frac2p m_j\ve\gamma_j(1+o(1)),
	\qquad j=2,\ldots,k.
	\]
	Hence $\eta_j = o (\gamma_j)$ as $\ve\to 0$, for all $j=2, \ldots , k$.
	To impose the Dirichlet boundary condition, let $H_j^\lambda$ solve
	\begin{equation*}\label{Hjdef}
		\Delta H_j^\lambda
		=
		\lambda\left[
		\log\frac1{(d_j^2+|x|^2)^2}-H_j^\lambda
		\right] \quad 
		\text{in }\Omega,\quad 
		H_j^\lambda
		=
		\log\frac1{(d_j^2+|x|^2)^2}
		\quad \text{on }\partial\Omega.
	\end{equation*}
	We define
	\begin{equation*}
		U_j(x)
		=W_j(x)
		+\frac{1}{p\gamma_j^{p-1}}
		\Big[
		2\log(\lambda p\gamma_j^p)-\log(8\mu_j^2)
		-\ve\gamma_j^p 
		-p\gamma_j^{p-1}\eta_j-H_j^\lambda(x)
		\Big].
		\label{Ujfirst}
	\end{equation*}
	where $W_j$ is defined in \eqref{defWj}.
	Using $d_j=\mu_j\rho_j$, this can be written more simply as
	\begin{equation}\label{juan}
		U_j(x)
		=
		\frac{1}{p\gamma_j^{p-1}}
		\left[
		\log\frac1{(d_j^2+|x|^2)^2}-H_j^\lambda(x)
		\right].
	\end{equation}
	In particular, if $|x|=o(d_i)$,
	\[
	p\gamma_i^{p-1}U_i(x)
	=
	-4\log d_i-h_\lambda(0)+o(1),
	\]
	which explains the sign in \eqref{defetaj}.
	Fix once and for all a number $\nu\in(0,1)$. The difference
	$H_j^\lambda-H_\lambda(\cdot,0)$ solves a uniformly coercive Dirichlet
	problem whose interior right-hand side is
	\[
	-2\lambda\log\left(1+\frac{d_j^2}{|x|^2}\right).
	\]
	For $q=2/(2-\nu)>1$, this right-hand side has $L^q(\Omega)$ norm
	$O(d_j^{2-\nu})$, while the boundary discrepancy is $O(d_j^2)$.
	Standard $W^{2,q}$ estimates therefore give
	\begin{equation*}\label{Hj-approx}
		\|H_j^\lambda-H_\lambda(\cdot,0)\|_{L^\infty(\Omega)}
		\le C_\nu d_j^{2-\nu}.
	\end{equation*}
	In particular,
	\begin{equation*}
		\label{Ujfuori}
		U_j(x)
		=
		\frac{1}{p\gamma_j^{p-1}}
		\bigl[G_\lambda(x,0)+o(1)\bigr]
		\qquad(\ve\to0)
	\end{equation*}
	uniformly on compact sets of $\Omega \setminus \{ 0 \}$.
	Moreover, \eqref{juan} gives
	\[
	(-\Delta-\lambda)U_j
	=
	\frac{b_j}{p}V_j>0,
	\qquad U_j=0\quad\hbox{on }\partial\Omega.
	\]
	Since $\lambda<\lambda_1(\Omega)$, the maximum principle for
	$-\Delta-\lambda$ yields
	\begin{equation*}\label{Upositive}
		U_j>0\quad\hbox{in }\Omega,
		\qquad U=\sum_{j=1}^kU_j>0.
	\end{equation*}
	
	We define the approximate solution to be
	\begin{equation}
		\label{approxisol}
		U(x)=\sum_{j=1}^kU_j(x).
	\end{equation}
	
	It remains to choose the shape parameters $\mu_j$. They are determined so
	that the $j$-th profile matches the total approximation on its physical
	scale $d_j=\mu_j\rho_j$. Equivalently, after writing $x=d_jy$, its leading
	nonconstant part is the normalized Liouville profile $w(y)$.
	To formulate this matching, introduce
	\begin{equation}
		\label{defGammaj}
		\Gamma_j
		=
		\sum_{i>j}\left(\frac{\gamma_j}{\gamma_i}\right)^{p-1},
		\qquad j=1,\ldots,k-1,
		\qquad
		\Gamma_k=0.
	\end{equation}
	For $j=1,\ldots,k-1$,
	\begin{equation*}
		\label{estGammaj}
		\Gamma_j
		=
		\left(\frac{\gamma_j}{\gamma_{j+1}}\right)^{p-1}(1+o(1)),
		\qquad j=1,\ldots,k-1.
	\end{equation*}
	
	\subsection{Exact matching of the constants}
	
	We impose the constant matching exactly.  The equations are solved
	successively in $j$: for $j=1$ one has $\eta_1=0$; once $d_1,\ldots,d_{j-1}$
	have been determined, $\eta_j$ is known and the equation below determines
	$\mu_j$.  This removes any circularity in the definition of the scales.
	
	Set
	\[
	G_j:=\gamma_j^p,\qquad L_j:=\log\mu_j^2.
	\]
	Using $d_j^2=\mu_j^2\rho_j^2$ and
	$\rho_j^{-2}=\lambda pG_j e^{G_j}$, we have
	\[
	-4\log d_j
	=
	-2L_j+2G_j+2\log(\lambda pG_j).
	\]
	
	We explain the matching condition before writing it down. At the $j$-th
	scale, $x=d_jy$, the sum of the projected bubbles has the form
	\[
	\frac{U(d_jy)}{\gamma_j}
	=
	1+\frac{1}{pG_j}
	\left\{w(y)-4\Gamma_j\log|y|+\hbox{constant term}\right\}
	+\hbox{smaller terms}.
	\]
	The constant term contains $L_j=\log\mu_j^2$ and is precisely the
	quantity $\mathcal B_j(L_j)$ introduced below. If one first expands this
	constant term in powers of $\ve$, an error which is small relative to
	$G_j$ need not be small after exponentiation, especially for the inner
	bubbles. We therefore keep this constant part unexpanded. Writing
	\[
	a_j(L)=1+\frac{\mathcal B_j(L)}{pG_j},
	\]
	the potentially large constant factor in the nonlinear term is
	\[
	\mu_j^2\exp\{G_j(a_j(L)^p-1)\}.
	\]
	The equation defining $L_j$ is chosen so that this factor is exactly one.
	The remaining prefactor $a_j=1+O(\ve)$ is harmless and is retained in the
	subsequent expansion.
	For a real variable $L$, define
	\begin{align}
		{\mathcal B}_j(L)
		:={}&-2L-\log8-h_\lambda(0)
		+2\log(\lambda pG_j)-\ve G_j-p\gamma_j^{p-1}\eta_j
		\nonumber\\
		&+\Gamma_j\bigl[
		-2L+2G_j+2\log(\lambda pG_j)-h_\lambda(0)
		\bigr],
		\label{Bjexact}
	\end{align}
	and
	\[
	a_j(L):=1+\frac{{\mathcal B}_j(L)}{pG_j}.
	\]
	We choose $L_j$ as the small solution of the exact matching equation
	\begin{equation}\label{defmuj}
		L_j+G_j\bigl[a_j(L_j)^p-1\bigr]=0.
	\end{equation}
	The quantity $a_j\gamma_j$ is the local height of the full approximation
	at the $j$-th core.  Once the small solution $L_j$ has been chosen, set
	\[
	\mu_j=e^{L_j/2},
	\qquad
	d_j=\mu_j\rho_j.
	\]
	
	\begin{lemma}[exact matching and robust scales]\label{lem:robust-scales}
		For every fixed $\delta>0$ and all $m\in[\delta,\delta^{-1}]^k$,
		equation \eqref{defmuj} has, for $\ve$ sufficiently small, a unique solution
		satisfying
		\[
		\frac{L_j}{G_j}=O(\ve)+O\!\left(\frac{\log G_j}{G_j}\right).
		\]
		Moreover
		\begin{align}
			\log\mu_j^2
			&=
			O\!\left(\ve G_j+\log G_j\right),
			\label{robust-mu}\\
			\log d_j^2
			&=
			-G_j+
			O\!\left(\ve G_j+\log G_j\right),
			\label{robust-d}
		\end{align}
		and consequently
		\[
		d_k\ll d_{k-1}\ll\cdots\ll d_1,
		\qquad
		\frac{d_j}{d_{j-1}}\to0.
		\]
	\end{lemma}
	
	\begin{proof}
		The construction is sequential in $j$.  For $j=1$ one has $\eta_1=0$.
		Assume inductively that $L_i$, $\mu_i$, and $d_i$ have already been
		constructed for $i<j$ and satisfy \eqref{robust-mu}--\eqref{robust-d}.
		Then, we clearly have 
		\[
		-4\log d_i-h_\lambda(0)
		=
		2G_i+O(\ve G_i+\log G_i),
		\]
		and therefore
		\[
		|\eta_j|
		\le
		C\sum_{i<j}\gamma_i
		\le C\gamma_{j-1}.
		\]
		Since
		\[
		\frac{\gamma_{j-1}}{\gamma_j}=\ve m_j,
		\]
		we obtain, uniformly in the parameter box \eqref{emme},
		\begin{equation*}\label{eta-rough}
			\frac{\eta_j}{\gamma_j}=O(\ve).
		\end{equation*}
		Also $\Gamma_j=O(\ve)$ directly from \eqref{defGammaj}.
		
		Write $\ell=L/G_j$.  Dividing \eqref{Bjexact} by $G_j$ gives
		\[
		\frac{{\mathcal B}_j(L)}{G_j}
		=
		-2(1+\Gamma_j)\ell
		-\ve-\frac{p\eta_j}{\gamma_j}
		+2\Gamma_j
		+
		O\!\left(\frac{\log G_j}{G_j}\right).
		\]
		Thus
		\[
		{\mathfrak F}_{j,\ve}(\ell)
		:=
		\ell+
		\left[
		1+\frac1p\frac{{\mathcal B}_j(G_j\ell)}{G_j}
		\right]^p-1
		\]
		converges, uniformly for $\ell$ in a fixed small interval, 
		to $\ell+(1-\ell)^2-1.$ Clearly, the zero $\ell=0$ is nondegenerate.  
		
		The implicit function theorem applied to ${\mathfrak F}_{j,\ve}(\ell)=0$, using \eqref{defmuj} gives a unique small branch and
		\[
		\ell
		=
		O(\ve)+O\!\left(\frac{\log G_j}{G_j}\right).
		\]
		This proves \eqref{robust-mu} for the $j$-th scale.  Since
		\[
		\log\rho_j^2=-G_j-\log(\lambda pG_j),
		\]
		formula \eqref{robust-d} follows as well.  The induction is complete.
		
		Finally,
		\[
		\log\frac{d_j^2}{d_{j-1}^2}
		=
		-G_j+G_{j-1}
		+O(\ve G_j+\log G_j).
		\]
		Because
		\[
		G_{j-1}=(\ve m_j)^pG_j=o(G_j),
		\]
		we obtain
		\[
		\log\frac{d_j^2}{d_{j-1}^2}=-(1+o(1))G_j.
		\]
		Hence $d_j/d_{j-1}\to0$.  In fact, since every $G_j$ dominates
		$|\log\ve|$, for every fixed $N>0$,
		\begin{equation*}\label{superalgebraic-separation}
			d_1=O(\ve^N),
			\qquad
			\frac{d_j}{d_{j-1}}=O(\ve^N),
			\quad j=2,\ldots,k.
		\end{equation*}
		This stronger form will be used repeatedly below.
	\end{proof}
	
	We write
	\[
	a_j:=a_j(L_j),
	\qquad
	\sigma_j:=a_j^{p-1}.
	\]
	By \eqref{defmuj},
	\begin{equation}\label{matchingidentity}
		\mu_j^2\exp\{G_j(a_j^p-1)\}=1.
	\end{equation}
	Thus the full constant displacement in the
	exponent is matched exactly, and no expansion of a quantity of size
	$\ve G_j$ inside an exponential will be used.
	
	\subsection{Positive weighted norms}
	
	We use positive weights adapted to the logarithmic radial variable.
	Set
	\[
	s=-\log|x|,\qquad s_j=-\log d_j.
	\]
	The numbers $s_j$ satisfy
	\[
	s_1\ll s_2\ll\cdots\ll s_k,
	\]
	whereas
	\[
	b_1\gg b_2\gg\cdots\gg b_k,
	\qquad
	\frac{b_{j+1}}{b_j}=(\ve m_{j+1})^{p-1},
	\]
	where $b_j$ is defined in \eqref{bjdef}.
	
	We first introduce the positive bridge function $\beta_\ve$.  It is the
	continuous function of $s$ such that
	\[
	\beta_\ve(s_j)=b_j,
	\]
	it is affine on each interval $[s_j,s_{j+1}]$, it is equal to $b_1$ for
	$s\le s_1$, and equal to $b_k$ for $s\ge s_k$.  We write
	\[
	\beta_\ve(x):=\beta_\ve(-\log|x|)
	\]
	for $x\ne0$, and $\beta_\ve(0)=b_k$.
	
	The affine interpolation in $s$ has a simple reason. If $F=F(s)$ is
	radial and $s=-\log r$, then
	\begin{equation}\label{radial-log-laplacian}
		\Delta F=e^{2s}F''(s).
	\end{equation}
	Thus an affine function of $s$ is harmonic in the corresponding planar
	annulus. Since the admissible size of the correction is $b_j$ at the
	$j$-th core and $b_{j+1}$ at the next one, $\beta_\ve$ is the harmonic
	interpolation of these two sizes. In particular, the norm below does not
	force the correction to remain at the size of either endpoint throughout
	a logarithmically very long bridge.
	
	The norm for the correction will be
	\begin{equation}\label{starnorm2}
		\|\phi\|_{**}
		:=
		\sup_{x\in\Omega}\frac{|\phi(x)|}{\beta_\ve(x)}.
	\end{equation}
	Thus, near the $j$-th bubble, $\|\phi\|_{**}=O(\ve)$ means precisely
	\[
	|\phi|=O(\ve\gamma_j^{1-p}),
	\]
	which is the scale required in order not to perturb the exponential
	profile.
	
	To define the norm for the error, let
	\[
	R_\ve:=|\log\ve|^2
	\]
	and choose $\chi\in C^\infty([0,\infty))$ such that
	\[0\le\chi\le1,\quad \chi=0\:\:\mbox{on}\:\:[0,1]\quad\mbox{and}\quad \chi=1\:\:\mbox{ on}\:\: [2,\infty).
	\]
	For $t\in\mathbb R$, set
	\[
	\tau_\ve(t):=\min\{|t|,2R_\ve\}
	\]
	and
	\begin{equation}\label{Thetaweight}
		\Theta_{j,\ve}(t)
		:=
		1+\tau_\ve(t)
		+\frac{1+\tau_\ve(t)^2}{\ve G_j}
		+\frac1\ve\chi\!\left(\frac{|t|}{R_\ve}\right).
	\end{equation}
	There is one additional contribution which cannot be bounded pointwise by
	the tails of the $V_j$'s in the middle of a very long bridge. In the
	outer bridge one must also preserve the cancellation
	$e^{U^p}-1$ once $U$ becomes small. Fix a spatial radius $r_0>0$, independent
	of $\ve$, so small that $\overline{B(0,2r_0)}\subset\Omega$ and all the
	local expansions used below are valid in $B(0,2r_0)$.
	
	Put
	\[
	\chi_{\rm tr}(x)
	:=
	\prod_{j=1}^k
	\chi\!\left(
	\frac{\left|\log(|x|/d_j)\right|}{R_\ve}
	\right)
	\]
	in $B(0,r_0)$ and extend it by zero outside $B(0,r_0)$. Note that $\chi_{\rm tr}$ is equal to one between consecutive bubbles and zero in the bubbles. 
	
	We define
	\begin{align*}
		{\mathcal D}^{\rm hi}_\ve(x)
		&:=
		\chi_{\rm tr}(x)\,
		\chi(U(x))\,\lambda U(x)e^{U(x)^p},
		\\
		{\mathcal D}^{\rm lo}_\ve(x)
		&:=
		\chi_{\rm tr}(x)\,
		[1-\chi(U(x))]\,
		\lambda U(x)\bigl(e^{U(x)^p}-1\bigr).
	\end{align*}
	The high-amplitude density is the one relevant between consecutive
	bubbles.  The low-amplitude density occurs only in the exterior bridge,
	where the cancellation $e^{U^p}-1$ is essential.
	
	We use the positive weight
	\begin{equation}\label{rho-positive}
		\begin{aligned}
			\varrho_\ve(x)
			:={}&
			b_1
			+\sum_{j=1}^k
			b_jV_j(x)\,
			\Theta_{j,\ve}\!\left(\log\frac{|x|}{d_j}\right)\\
			&+\ve^{-2}{\mathcal D}^{\rm hi}_\ve(x)
			+\ve^{-1}{\mathcal D}^{\rm lo}_\ve(x).
		\end{aligned}
	\end{equation}
	At $x=0$, the logarithmic factors in $\varrho_\ve$ are understood by
	their limits as $|x|\to0$; this defines $\varrho_\ve(0)>0$.
	The associated norm is
	\begin{equation*}\label{starnorm}
		\|h\|_*:=
		\sup_{x\in\Omega}\frac{|h(x)|}{\varrho_\ve(x)}.
	\end{equation*}
	Both $\|\cdot\|_*$ and $\|\cdot\|_{**}$ depend on $(\ve,m)$ through
	the scales and weights; this dependence is suppressed from the notation.
	For fixed $\ve$, the norms corresponding to $m$ in the compact parameter
	box are uniformly equivalent.  All terms in \eqref{rho-positive} are
	non-negative.
	
	For the estimates in this section we write
	\begin{equation*}\label{positive-nonlinearity}
		f_+(s):=\lambda s e^{s^p},\qquad s>0.
	\end{equation*}
	Since the approximate solution $U$ is positive, this is the only branch of
	the nonlinearity used before the projected construction.
	
	\begin{lemma}[uniform core expansion]\label{lem:core-expansion}
		Let
		\[
		t=\log\frac{|x|}{d_j},\qquad x=d_jy.
		\]
		Uniformly for
		\[
		|t|\le2R_\ve,\qquad j=1,\ldots,k,
		\]
		one has
		\begin{equation}\label{coarse-local-U}
			\frac{U(d_jy)}{\gamma_j}
			=
			a_j+
			\frac{w(y)-4\Gamma_j\log|y|}{pG_j}
			+r_{j,\ve}^{\,c}(y),
			\qquad
			G_j|r_{j,\ve}^{\,c}(y)|=o(1).
		\end{equation}
		Consequently,
		\begin{align}
			\lambda Ue^{U^p}
			&=
			\frac{b_j}{p}V_j
			\left[
			1+
			O\!\left(
			\ve(1+|t|)
			+\frac{1+t^2}{G_j}
			\right)
			\right],
			\label{core-f}\\
			f_+'(U)-\lambda
			&=
			V_j
			\left[
			1+
			O\!\left(
			\ve(1+|t|)
			+\frac{1+t^2}{G_j}
			\right)
			\right].
			\label{core-fprime}
		\end{align}
		All remainders are uniform in the parameter box.
	\end{lemma}
	
	\begin{proof}
		Using \eqref{juan}, the contribution of an outer bubble $i<j$ at
		$x=d_jy$ is
		\[
		p\gamma_i^{p-1}U_i(d_jy)=-4\log d_i-h_\lambda(0)
		+
		O\!\left(
		\frac{d_j^2|y|^2}{d_i^2}
		+d_j|y|+d_i^{2-\nu}
		\right),
		\]
		whereas an inner bubble $i>j$ contributes
		\[
		p\gamma_i^{p-1}U_i(d_jy)=-4\log(d_j|y|)-h_\lambda(0)
		+
		O\!\left(
		\frac{d_i^2}{d_j^2|y|^2}
		+d_j|y|+d_i^{2-\nu}
		\right).
		\]
		The constants from $i<j$ cancel by \eqref{defetaj}; the logarithmic
		contributions from $i>j$ produce $-4\Gamma_j\log|y|$.  Since
		$|t|\le2R_\ve$ and \eqref{superalgebraic-separation} holds, all remaining
		terms are $o(G_j^{-1})$ after division by $\gamma_j$, which proves
		\eqref{coarse-local-U}.
		
		Now write
		\[
		U(d_jy)
		=
		\gamma_j
		\left[
		a_j+
		\frac{w(y)-4\Gamma_j\log|y|}{pG_j}
		+r_{j,\ve}^{\,c}(y)
		\right].
		\]
		The exact matching identity
		\[
		\mu_j^2e^{G_j(a_j^p-1)}=1
		\]
		removes the complete constant part of the exponential.  Since
		$a_j=1+O(\ve)$, $\Gamma_j=O(\ve)$ and
		$|w(y)|\le C(1+|t|)$ in this region, Taylor expansion only in the
		non-constant variable gives \eqref{core-f}.  Differentiating
		$f_+(s)=\lambda s e^{s^p}$, that is 
		\[
		f'_+(s)=\lambda (p s^p+1)e^{s^p},
		\]
		at the positive value $U$ gives
		\eqref{core-fprime}; the additional term $\lambda e^{U^p}$ is smaller by
		a factor $O(G_j^{-1})$.
	\end{proof}

	\begin{lemma}[shape of the nonlinear density on the bridges]
		\label{lem:localization}
		Let $\omega\in\mathbb S^1$ and write
		\[
		U_\omega(s):=U(e^{-s}\omega).
		\]
		For $1\le j\le k-1$ set
		\[
		{\mathcal J}_j=[s_j,s_{j+1}],
		\qquad
		{\mathcal J}_j^{\rm tr}=[s_j+2R_\ve,s_{j+1}-2R_\ve].
		\]
		For all sufficiently small $\ve$, uniformly in the parameter box and in
		$\omega$, the following properties hold.
		
		\begin{enumerate}
			\item The function
			\[
			{\mathcal H}_\omega(s)
			:=-2s+\log(\lambda U_\omega(s))+U_\omega(s)^p
			\]
			is strictly convex on the interval
			\[
			s_j+R_\ve\le s\le s_{j+1}-R_\ve.
			\]
			In particular it has at most one critical point on this interval,
			necessarily a minimum.
			
			\item There are constants $c,C>0$ such that
			\begin{equation}\label{Dtransition-small}
				\sup_{\omega\in\mathbb S^1}\sup_{s\in{\mathcal J}_j^{\rm tr}}
				e^{-2s}\lambda U_\omega(s)e^{U_\omega(s)^p}
				\le C e^{-cR_\ve}(b_j+b_{j+1}).
			\end{equation}
			
			\item For every $j=1,\ldots,k-1$ and every
			$s\in{\mathcal J}_j=[s_j,s_{j+1}]$,
			\begin{equation}\label{betaUbound}
				\beta_\ve(s)U_\omega(s)^{p-1}\le \frac{C}{\ve}.
			\end{equation}
			
			\item Let $s_0=-\log r_0$.  In the outer bridge
			$[s_0,s_1-2R_\ve]$, restricted to $U_\omega\ge1$,
			${\mathcal H}_\omega$ is strictly convex and
			\begin{equation}\label{outer-high-density}
				e^{-2s}\lambda U_\omega(s)e^{U_\omega(s)^p}
				\le C\left[b_1e^{-cR_\ve}+e^{-c/b_1}\right].
			\end{equation}
		\end{enumerate}
	\end{lemma}
	
	\begin{proof}
		Differentiating \eqref{juan}--\eqref{approxisol} along a ray gives, uniformly in $\omega$,
		\begin{equation}\label{Usbridge}
			\partial_sU_\omega(s)
			=
			\frac4p\sum_{i=1}^k b_i\frac{e^{-2s}}{d_i^2+e^{-2s}}
			+O\!\left(e^{-s}\sum_i b_i\right).
		\end{equation}
		The last term comes from the smooth functions $H_i^\lambda$; the reflection
		symmetries actually improve it to quadratic order at the origin, but the
		displayed estimate is sufficient.
		
		If $s_j+R_\ve\le s\le s_{j+1}-R_\ve$, the terms $i\le j$ are
		exponentially small, whereas
		\[
		\sum_{i\ge j+1}b_i=b_{j+1}(1+O(\ve)).
		\]
		Using \eqref{Usbridge} and the scale separation, we first obtain the
		first-derivative estimate
		\begin{equation*}\label{Us-central}
			\partial_sU_\omega
			=
			\frac4p b_{j+1}(1+O(\ve))+o(b_{j+1}).
		\end{equation*}
		Differentiating \eqref{Usbridge} with respect to $s$ gives, in addition,
		the following second-derivative estimate: for every fixed $N>0$,
		\begin{equation}\label{Uss-super-small}
			\sup_{\omega}
			\sup_{s_j+R_\ve\le s\le s_{j+1}-R_\ve}
			|\partial_{ss}U_\omega(s)|
			=
			o(b_{j+1}^{\,N}).
		\end{equation}
		Indeed, after this differentiation every bubble term
		contains a factor $e^{-2R_\ve}$, while the smooth part contains the
		physical factor $e^{-s}$.  Since
		$R_\ve=|\log\ve|^2$ and the scale separation
		\eqref{superalgebraic-separation} is faster than every algebraic power of
		$\ve$, both contributions are $o(b_{j+1}^{\,N})$ for every fixed $N$.
		Moreover $U_\omega\ge c\gamma_j$ in the central bridge.  A direct
		calculation yields
		\begin{align*}
			{\mathcal H}_\omega''
			={}&\frac{\partial_{ss}U_\omega}{U_\omega}
			-\frac{\partial_sU_\omega^2}{U_\omega^2}
			+p(p-1)U_\omega^{p-2}\partial_sU_\omega^2
			+pU_\omega^{p-1}\partial_{ss}U_\omega.
		\end{align*}
		The negative term is $o(b_{j+1}^2)$ because
		$U_\omega\ge c\gamma_j\to\infty$.  Moreover
		$U_\omega^{p-1}\le Cb_{j+1}^{-1}$, so
		\eqref{Uss-super-small} with, for instance, $N=3$ shows that both terms
		containing $\partial_{ss}U_\omega$ are $o(b_{j+1}^2)$.  On the other hand,
		\[
		p(p-1)U_\omega^{p-2}\partial_sU_\omega^2\ge cb_{j+1}^2.
		\]
		Thus ${\mathcal H}_\omega''>0$, uniformly in $\omega$.
		
		Lemma \ref{lem:core-expansion} near $s_j$ gives
		\[
		e^{-2s}\lambda U_\omega e^{U_\omega^p}
		\le Cb_j e^{-c|s-s_j|}
		\]
		for $|s-s_j|\le2R_\ve$, uniformly in $\omega$; the analogous estimate
		holds near $s_{j+1}$.  Strict convexity shows that the density has no
		interior maximum in the central bridge and proves \eqref{Dtransition-small}.
		
		On ${\mathcal J}_j$,
		\[
		U_\omega\le C\gamma_{j+1},\qquad \beta_\ve\le b_j.
		\]
		Since $\gamma_j/\gamma_{j+1}=\ve m_{j+1}$, this gives
		\[
		\beta_\ve U_\omega^{p-1}
		\le C\left(\frac{\gamma_{j+1}}{\gamma_j}\right)^{p-1}
		\le \frac C\ve.
		\]
		
		Finally consider the outer bridge.  Away from the first core,
		\eqref{Usbridge} gives on the set $U_\omega\ge1$
		\[
		\partial_sU_\omega=\frac4p b_1(1+O(\ve))+o(b_1).
		\]
		Exactly as above, for every fixed $N>0$,
		\begin{equation}\label{outer-Uss-super-small}
			\sup_{\omega}
			\sup_{\substack{s_0\le s\le s_1-R_\ve\\U_\omega(s)\ge1}}
			|\partial_{ss}U_\omega(s)|
			=
			o(b_1^{\,N}).
		\end{equation}
		Since
		\[
		p(p-1)U_\omega^{p-2}-U_\omega^{-2}\ge c>0
		\quad\hbox{for }U_\omega\ge1,
		\]
		and $U_\omega^{p-1}\le Cb_1^{-1}$ on the outer bridge,
		\eqref{outer-Uss-super-small} with $N=3$ shows that the terms containing
		$\partial_{ss}U_\omega$ are negligible.  The same calculation therefore gives
		${\mathcal H}_\omega''>0$.  At the core-side
		endpoint the matching expansion gives $Cb_1e^{-cR_\ve}$.  At the other
		endpoint of the high-amplitude region, where $U_\omega=1$, the estimate
		\[
		U_\omega(s)\le Cb_1(1+s)
		\]
		implies $s\ge c/b_1$, so the density is at most $Ce^{-c/b_1}$.  Convexity
		proves \eqref{outer-high-density}.
	\end{proof}
	
	For $0<r<r_0$ define the radial majorant
	\begin{equation*}\label{radial-rho-majorant}
		\overline\varrho_\ve(r):=\sup_{\omega\in\mathbb S^1}\varrho_\ve(r\omega).
	\end{equation*}
	All estimates above are uniform in $\omega$, so replacing $\varrho_\ve$ by
	$\overline\varrho_\ve$ does not change their size.
	
	We next construct the supersolution which propagates estimates from one
	core to the next. Formula~\eqref{radial-log-laplacian} reduces the radial
	part of the problem on a bridge to a one-dimensional equation in $s$.
	The affine function $\beta_\ve$ supplies the correct boundary sizes, but
	$L\beta_\ve={\mathcal V}_\ve\beta_\ve$ has the wrong sign. We therefore add a
	non-negative correction $P_j$ solving a one-dimensional Dirichlet problem.
	The source of that problem is concentrated near the two ends of the bridge.
	Consequently it is the first moment of the source, rather than only its
	total mass, which controls $P_j$. Lemma~\ref{lem:bridge-moments} records
	these moment estimates; Lemma~\ref{lem:affine-bridge} then gives
	$P_j=o(\beta_\ve)$ and a positive supersolution of the required size.
	
	\begin{lemma}[bridge moments]\label{lem:bridge-moments}
		Fix $R>1$.  For $j=2,\ldots,k$ set
		\[
		A_j=s_{j-1}+\log R,
		\qquad
		B_j=s_j-\log R,
		\qquad
		I_j^R=[A_j,B_j].
		\]
		Let
		\[
		{\mathcal V}_\ve=\lambda+\sum_{i=1}^kV_i
		\]
		and
		\[
		g_j(s)
		=
		e^{-2s}\left[
		2{\mathcal V}_\ve(e^{-s})\beta_\ve(s)+\overline\varrho_\ve(e^{-s})
		\right].
		\]
		Then $g_j$ admits a decomposition
		\[
		g_j=g_j^{\,L}+g_j^{\,R}+g_j^{\,M},
		\qquad g_j^{\,L},g_j^{\,R},g_j^{\,M}\ge0,
		\]
		where $g_j^{\,L}$ is supported in the left half of $I_j^R$,
		$g_j^{\,R}$ in the right half, and
		\begin{align}
			\int_{A_j}^{B_j}(t-A_j)g_j^{\,L}(t)\,dt
			&\le \eta_R\,\beta_\ve(A_j),
			\label{leftmoment}\\
			\int_{A_j}^{B_j}(B_j-t)g_j^{\,R}(t)\,dt
			&\le \eta_R\,\beta_\ve(B_j),
			\label{rightmoment}\\
			(B_j-A_j)^2\|g_j^{\,M}\|_\infty
			&\le o_\ve(1)\min\{\beta_\ve(A_j),\beta_\ve(B_j)\},
			\label{middlemoment}
		\end{align}
		with
		\[
		\lim_{R\to\infty}\limsup_{\ve\to0}\eta_R=0.
		\]
		
		On the outer bridge
		\[
		I_1^R=[s_0,s_1-\log R],
		\qquad s_0=-\log r_0,
		\]
		the same three estimates hold with $\beta_\ve\equiv b_1$ and with a
		constant $\eta_R^{\rm out}(r_0)$ satisfying
		\begin{equation}\label{eta-outer}
			\limsup_{\ve\to0}\eta_R^{\rm out}(r_0)
			\le Cr_0^2+\eta_R.
		\end{equation}
		Thus $\eta_R^{\rm out}(r_0)$ can be made arbitrarily small by choosing
		first $r_0$ small and then $R$ large.
	\end{lemma}
	
	\begin{proof}
		We give the decomposition explicitly.  Let
		\[
		s_j^*:=\frac{A_j+B_j}{2},
		\qquad
		\ell_j:=B_j-A_j.
		\]
		On an inter-bubble bridge $I_j^R$, $j\ge2$, the low-amplitude transition
		density vanishes for small $\ve$.  Since all the remaining terms in
		$\varrho_\ve$ are radial except for ${\mathcal D}^{\rm hi}_\ve$, write
		\begin{align}
			g_j^{\rm rad}(s)
			:={}&e^{-2s}\left[
			2{\mathcal V}_\ve(e^{-s})\beta_\ve(s)+b_1
			+\sum_{i=1}^k b_iV_i(e^{-s})
			\Theta_{i,\ve}(s_i-s)
			\right],
			\label{qj-bridge}\\
			g_j^{\rm tr}(s)
			:={}&\ve^{-2}e^{-2s}
			\sup_{\omega\in\mathbb S^1}
			{\mathcal D}^{\rm hi}_\ve(e^{-s}\omega).
			\label{hj-bridge}
		\end{align}
		Then $g_j=g_j^{\rm rad}+g_j^{\rm tr}$.  We set
		\begin{equation*}\label{g-decomposition}
			g_j^{\,L}:={\bf 1}_{[A_j,s_j^*]}g_j^{\rm rad},
			\qquad
			g_j^{\,R}:={\bf 1}_{(s_j^*,B_j]}g_j^{\rm rad},
			\qquad
			g_j^{\,M}:=g_j^{\rm tr}.
		\end{equation*}
		This is the decomposition appearing in the statement.
		
		We first record two elementary facts about the affine bridge weight.
		From \eqref{robust-d}, \eqref{explicitgamma}, and
		$b_j=\gamma_j^{1-p}$,
		\begin{equation}\label{beta-endpoints}
			\beta_\ve(A_j)=b_{j-1}(1+o(1)),
			\qquad
			\beta_\ve(B_j)=b_j(1+o(1)),
		\end{equation}
		and
		\begin{equation*}\label{beta-slope}
			\frac{|\beta_\ve'|}{b_j}
			=
			O\!\left(
			\frac{b_{j-1}/b_j}{s_j-s_{j-1}}
			\right)
			=o(1).
		\end{equation*}
		Indeed $b_{j-1}/b_j$ is only an algebraic power of $\ve^{-1}$, whereas
		$s_j-s_{j-1}=\frac12G_j(1+o(1))$ grows faster than that power.  It follows
		that
		\begin{equation}\label{beta-left-right}
			\beta_\ve(s)\le \beta_\ve(A_j)
			\quad(A_j\le s\le s_j^*),
		\end{equation}
		and, on the right half,
		\begin{equation}\label{beta-right-growth}
			\beta_\ve(s)
			\le
			C\beta_\ve(B_j)\,[1+B_j-s].
		\end{equation}
		
		We now estimate the left moment.  The bubble adjacent to the left endpoint
		is the $(j-1)$-st one.  Put $\tau=s-s_{j-1}$.  Since
		$A_j=s_{j-1}+\log R$, one has $\tau\ge\log R$, and
		\begin{equation}\label{sechidentity}
			e^{-2s}V_{j-1}(e^{-s})
			=2\,\operatorname{sech}^2\tau
			\le Ce^{-2\tau}.
		\end{equation}
		By \eqref{beta-left-right}, \eqref{beta-endpoints}, and
		$1/(\ve G_i)\le C$ for every $i$, the contribution of this bubble to the
		left moment is bounded by
		\begin{align}
			C\beta_\ve(A_j)
			\bigg[&\int_{\log R}^{\infty}
			\tau\,(1+\tau+\tau^2)e^{-2\tau}\,d\tau
			\nonumber\\
			&+\ve^{-1}\int_{R_\ve}^{\infty}
			\tau e^{-2\tau}\,d\tau\bigg].
			\label{left-neighbour-estimate}
		\end{align}
		The first integral tends to zero as $R\to\infty$, uniformly in $\ve$;
		the second is $o(\ve^N)$ for every fixed $N$ because
		$R_\ve=|\log\ve|^2$.
		
		All bubbles with $i<j-1$ are farther to the left, while those with
		$i\ge j$ are separated from the left half by at least a fixed fraction of
		the logarithmic distance $s_j-s_{j-1}$.  Using
		\eqref{sechidentity} with the corresponding centers and the
		superalgebraic separation \eqref{superalgebraic-separation}, the sum of
		all these contributions is
		\begin{equation}\label{left-other-bubbles}
			o_\ve(1)\,\beta_\ve(A_j).
		\end{equation}
		Finally, the terms $b_1e^{-2s}$ and
		$2\lambda e^{-2s}\beta_\ve(s)$ satisfy
		\begin{align}
			\int_{A_j}^{s_j^*}(s-A_j)e^{-2s}
			[b_1+2\lambda\beta_\ve(s)]\,ds
			&\le
			C d_{j-1}^2
			[b_1+\beta_\ve(A_j)]
			\nonumber\\
			&=o_\ve(1)\beta_\ve(A_j).
			\label{left-constant-terms}
		\end{align}
		Combining \eqref{left-neighbour-estimate}--\eqref{left-constant-terms}
		proves \eqref{leftmoment}.
		
		The right moment is analogous, but it is useful to spell out the only
		minor difference.  The adjacent bubble is now the $j$-th one.  Put
		$\tau=s_j-s$; then $\tau\ge\log R$ on the right half and
		$B_j-s\le\tau$.  By \eqref{beta-right-growth},
		\[
		\beta_\ve(s)
		\le C\beta_\ve(B_j)(1+\tau),
		\qquad
		b_j\le C\beta_\ve(B_j).
		\]
		Therefore the contribution of $V_j$ is bounded by
		\begin{align}
			C\beta_\ve(B_j)
			\bigg[&\int_{\log R}^{\infty}
			\tau(1+\tau)^2(1+\tau+\tau^2)e^{-2\tau}\,d\tau
			\nonumber\\
			&+\ve^{-1}\int_{R_\ve}^{\infty}
			\tau(1+\tau)e^{-2\tau}\,d\tau\bigg]
			\le \eta_R\beta_\ve(B_j).
			\label{right-neighbour-estimate}
		\end{align}
		The other bubbles and the constant terms are superalgebraically smaller
		on the right half, exactly as above.  This gives \eqref{rightmoment}.
		
		It remains to estimate $g_j^{\,M}=g_j^{\rm tr}$.  Lemma~\ref{lem:localization},
		with index $j-1$, gives on the support of the transition density
		\[
		e^{-2s}
		\sup_{\omega}{\mathcal D}^{\rm hi}_\ve(e^{-s}\omega)
		\le
		Ce^{-cR_\ve}(b_{j-1}+b_j).
		\]
		Hence
		\begin{equation}\label{middle-direct}
			\ell_j^2\|g_j^{\,M}\|_\infty
			\le
			C\ell_j^2\ve^{-2}e^{-cR_\ve}(b_{j-1}+b_j).
		\end{equation}
		Both $\ell_j$ and $b_{j-1}/b_j$ grow at most algebraically in
		$\ve^{-1}$, while $e^{-cR_\ve}=e^{-c|\log\ve|^2}$ decays faster than
		every algebraic power.  Since the smaller endpoint value is comparable to
		$b_j$, \eqref{middle-direct} is precisely \eqref{middlemoment}.
		This completes the proof for $j=2,\ldots,k$.
		
		We finally treat the outer bridge.  Put
		\[
		A_1=s_0,
		\qquad
		B_1=s_1-\log R,
		\qquad
		M_1=\frac{A_1+B_1}{2}.
		\]
		Here $\beta_\ve\equiv b_1$.  In the definition of $q_1$ we include, in
		addition to the terms in \eqref{qj-bridge}, the low-amplitude contribution
		\[
		\ve^{-1}e^{-2s}
		\sup_{\omega}{\mathcal D}^{\rm lo}_\ve(e^{-s}\omega),
		\]
		and we again take the high-amplitude term as $g_1^{\,M}$.  The $V_1$
		contribution near $B_1$ gives the same $\eta_Rb_1$ estimate as
		\eqref{right-neighbour-estimate}; all bubbles $i\ge2$ are
		superalgebraically smaller.  At the left endpoint,
		\begin{equation}\label{outer-constant-moment}
			\int_{A_1}^{M_1}(s-A_1)e^{-2s}b_1\,ds
			\le Cr_0^2b_1.
		\end{equation}
		Moreover, by the estimate already proved for the low-amplitude density,
		\begin{equation}\label{outer-low-moment}
			\ve^{-1}
			\int_{A_1}^{B_1}(1+s)e^{-2s}
			\sup_{\omega}{\mathcal D}^{\rm lo}_\ve(e^{-s}\omega)\,ds
			=o(b_1).
		\end{equation}
		The high-amplitude term satisfies the middle estimate by
		\eqref{outer-high-density} and the same superalgebraic argument used in
		\eqref{middle-direct}.  Equations
		\eqref{outer-constant-moment}--\eqref{outer-low-moment} therefore give
		\eqref{eta-outer}.  The proof is complete.
	\end{proof}
	
	\begin{lemma}[affine bridge supersolution]\label{lem:affine-bridge}
		Choose first $r_0>0$ small and then $R>1$ large so that, with the notation
		of Lemma~\ref{lem:bridge-moments},
		\[
		\widehat\eta_R
		:=
		\max\{\eta_R,\eta_R^{\rm out}(r_0)\}
		<\frac12.
		\]
		On every bridge there exists a non-negative radial function $P_j$, zero at
		the two endpoints, such that
		\begin{equation}\label{Pbound}
			0\le P_j(s)\le \widehat\eta_R\beta_\ve(s),
		\end{equation}
		and, with
		\[
		{\mathcal S}_j(s):=\beta_\ve(s)+P_j(s),
		\]
		one has
		\begin{equation}\label{Ssuper}
			L{\mathcal S}_j
			\le
			-\varrho_\ve
			-\frac12{\mathcal V}_\ve\beta_\ve
			<0
			\quad\hbox{on the corresponding radial bridge.}
		\end{equation}
	\end{lemma}
	
	\begin{proof}
		Let $P_j$ solve
		\[
		-P_j''=g_j,
		\qquad
		P_j(A_j)=P_j(B_j)=0,
		\]
		where $g_j$ is given in Lemma \ref{lem:bridge-moments}.  Positivity follows
		from the one-dimensional maximum principle.  The Green function on
		$[A_j,B_j]$ is
		\[
		G(s,t)=
		\begin{cases}
			\dfrac{(t-A_j)(B_j-s)}{B_j-A_j},&t\le s,\\[2mm]
			\dfrac{(s-A_j)(B_j-t)}{B_j-A_j},&t\ge s.
		\end{cases}
		\]
		For every $s,t\in[A_j,B_j]$ the Green function satisfies the two
		bounds
		\begin{equation*}\label{Green-affine-bounds}
			G(s,t)
			\le
			\frac{B_j-s}{B_j-A_j}(t-A_j),
			\qquad
			G(s,t)
			\le
			\frac{s-A_j}{B_j-A_j}(B_j-t).
		\end{equation*}
		Indeed, one inequality is an equality on each side of $t=s$, and the
		other follows from $t\le s$ or $t\ge s$, respectively.  Applying the first
		bound to $g_j^{\,L}$ and the second to $g_j^{\,R}$ gives
		\[
		P_j^{L}(s)
		\le
		\widehat\eta_R\frac{B_j-s}{B_j-A_j}\beta_\ve(A_j),
		\qquad
		P_j^{R}(s)
		\le
		\widehat\eta_R\frac{s-A_j}{B_j-A_j}\beta_\ve(B_j).
		\]
		For the middle contribution,
		\[
		\int_{A_j}^{B_j}G(s,t)\,dt
		=
		\frac{(s-A_j)(B_j-s)}2
		\le\frac{(B_j-A_j)^2}{8},
		\]
		so \eqref{middlemoment} gives $P_j^M=o_\ve(1)\beta_\ve(s)$ because
		$\beta_\ve(s)\ge\min\{\beta_\ve(A_j),\beta_\ve(B_j)\}$.  Finally,
		$\beta_\ve$ is affine on the interval:
		\[
		\beta_\ve(s)
		=
		\frac{B_j-s}{B_j-A_j}\beta_\ve(A_j)
		+
		\frac{s-A_j}{B_j-A_j}\beta_\ve(B_j).
		\]
		After increasing $R$ and decreasing $\ve_0$ if necessary, we conclude
		\[
		P_j\le\widehat\eta_R\beta_\ve.
		\]
		The choice of $r_0$ and $R$ gives \eqref{Pbound}.
		
		For a radial function $F(s)$,
		\[
		\Delta F=e^{2s}F''.
		\]
		Since $\beta_\ve''=0$ on a bridge and
		\[
		-P_j''
		=
		e^{-2s}\bigl(2{\mathcal V}_\ve\beta_\ve+\overline\varrho_\ve\bigr),
		\]
		we obtain
		\begin{align*}
			L{\mathcal S}_j
			&=
			-2{\mathcal V}_\ve\beta_\ve-\overline\varrho_\ve
			+{\mathcal V}_\ve(\beta_\ve+P_j)\\
			&=
			-{\mathcal V}_\ve\beta_\ve-\overline\varrho_\ve+{\mathcal V}_\ve P_j\\
			&\le
			-(1-\widehat\eta_R){\mathcal V}_\ve\beta_\ve-\overline\varrho_\ve
			\le -(1-\widehat\eta_R){\mathcal V}_\ve\beta_\ve-\varrho_\ve,
		\end{align*}
		which implies \eqref{Ssuper}.
		
		We record the maximum-principle consequence that will be used later.  If
		$D$ is one of the bridges, $\psi\in C^2(D)\cap C(\overline D)$ satisfies
		\[
		L\psi\le0\quad\hbox{in }D,
		\qquad
		\psi\ge0\quad\hbox{on }\partial D,
		\]
		then $\psi\ge0$ in $D$.  Indeed, if $\psi/{\mathcal S}_j$ had a negative
		interior minimum, then at that point
		\[
		L\psi
		=
		{\mathcal S}_j\Delta\!\left(\frac\psi{{\mathcal S}_j}\right)
		+
		\frac\psi{{\mathcal S}_j}L{\mathcal S}_j
		>0,
		\]
		a contradiction.  Thus the strict positive supersolution restores the
		maximum principle despite the positive zeroth-order potential.
	\end{proof}
	
	\begin{proposition}[error estimate]\label{prop:error}
		For all sufficiently small $\ve$,
		\begin{equation}\label{estRR}
			\|E\|_*\le C\ve.
		\end{equation}
		where $E$ is defined in \eqref{error}.
		Furthermore, if
		\[
		q_\ve(x):=f_+'(U)-\lambda-\sum_{j=1}^kV_j,
		\]
		then
		\begin{equation}\label{linear}
			\|q_\ve\phi\|_*
			\le o(1)\,\|\phi\|_{**}
		\end{equation}
		uniformly in the parameter box.
	\end{proposition}
	
	\begin{proof}
		In every logarithmic core,
		Lemma \ref{lem:core-expansion}, together with the contributions of the
		other bubbles and the definition of \eqref{Thetaweight}, gives
		\[
		|E|\le C\ve\varrho_\ve,
		\qquad
		|q_\ve|\beta_\ve=o(1)\varrho_\ve.
		\]
		
		On the complement of the cores, $\chi_{\rm tr}=1$ except in the smoothing layers, which are still covered
		by the core estimate.  Using the exact cancellation in the definition of
		$E$,
		\[
		|E|
		\le
		\frac1p\sum_i b_iV_i
		+
		\lambda U(e^{U^p}-1).
		\]
		The far-tail term $\ve^{-1}$ in \eqref{Thetaweight} gives
		\[
		\sum_i b_iV_i\le C\ve\varrho_\ve.
		\]
		For the nonlinear term we split according to $\chi(U)$.  On the
		high-amplitude part,
		\[
		\chi(U)\lambda U(e^{U^p}-1)
		\le {\mathcal D}^{\rm hi}_\ve
		\le
		\ve\,\ve^{-2}{\mathcal D}^{\rm hi}_\ve,
		\]
		whereas on the low-amplitude part the term is exactly
		${\mathcal D}^{\rm lo}_\ve$ and hence
		\[
		{\mathcal D}^{\rm lo}_\ve
		=
		\ve\,(\ve^{-1}{\mathcal D}^{\rm lo}_\ve).
		\]
		This proves \eqref{estRR}.
		
		For the linear remainder, write
		\[
		q_\ve=f_+'(U)-\lambda-\sum_iV_i.
		\]
		The bubble-potential terms and the bounded $\lambda$ contribution are
		$o(1)\varrho_\ve/\beta_\ve$ exactly as in the core estimate.
		On the high-amplitude transition region,
		\eqref{betaUbound} gives
		\[
		\chi(U)|q_\ve|\beta_\ve
		\le
		\frac C\ve{\mathcal D}^{\rm hi}_\ve
		=
		C\ve\,(\ve^{-2}{\mathcal D}^{\rm hi}_\ve).
		\]
		It remains to treat the low-amplitude part, which occurs only in the outer
		bridge.  There $\beta_\ve=b_1$, $0\le U\le2$, and
		\[
		|f_+'(U)-\lambda|\le CU^p.
		\]
		Put $y=1-\chi(U)\in[0,1]$.  Since
		$e^{U^p}-1\ge cU^p$ on $[0,2]$,
		\[
		{\mathcal D}^{\rm lo}_\ve\ge c\,yU^{p+1}
		\]
		in the transition region.  The elementary estimate
		\[
		\frac{y b_1U^p}{b_1+\ve^{-1}yU^{p+1}}
		\le C(\ve b_1)^{p/(p+1)}=o(1)
		\]
		is uniform in $y$ and $U$.  Thus
		\[
		[1-\chi(U)]|q_\ve|\beta_\ve
		=o(1)\varrho_\ve.
		\]
		It remains to record explicitly the region
		$\Omega\setminus B(0,r_0)$, where $\chi_{\rm tr}=0$.  From
		\eqref{juan} and the superalgebraic smallness of the $d_j$'s,
		\begin{equation*}\label{exterior-U-small}
			|U(x)|\le Cb_1,
			\qquad
			\sum_{j=1}^k b_jV_j(x)=o(\ve b_1)
			\quad\hbox{uniformly for }|x|\ge r_0.
		\end{equation*}
		Moreover
		\[
		b_1^p
		=
		\gamma_1^{p(1-p)}
		=
		O\!\left(\frac{\ve}{|\log\ve|}\right).
		\]
		Hence
		\[
		\lambda U(e^{U^p}-1)
		=
		O(b_1^{p+1})
		=
		o(\ve b_1),
		\]
		and therefore
		\[
		|E|\le C\ve b_1\le C\ve\varrho_\ve
		\quad\hbox{on }\Omega\setminus B(0,r_0).
		\]
		Similarly,
		\[
		|f_+'(U)-\lambda|
		\le C|U|^p=O(b_1^p),
		\]
		and the bubble potentials are superalgebraically small there.  Since
		$\beta_\ve=b_1$ in the exterior,
		\[
		|q_\ve|\beta_\ve=o(b_1)\le o(1)\varrho_\ve.
		\]
		This completes the proof of both \eqref{estRR} and \eqref{linear}.
	\end{proof}
	
	\section{Projected nonlinear problem and finite-dimensional reduction}
	
	We look for a solution in the form
	\[
	u=U+\phi,
	\]
	with $U$ defined in \eqref{approxisol}.  The correction is required to be
	small in the norm \eqref{starnorm2}.  We work throughout in the symmetric
	class
	\begin{equation}\label{simmetria}
		\phi(x_1,x_2)
		=
		\phi(-x_1,x_2)
		=
		\phi(x_1,-x_2).
	\end{equation}
	
	For the fixed-point argument it is convenient to extend the nonlinearity to
	the whole real line:
	\begin{equation*}\label{fextension}
		f(s):=\lambda s\exp\{(s_+)^p\},
		\qquad
		s_+:=\max\{s,0\}.
	\end{equation*}
	Since $U>0$, this extension does not change any of the expansions obtained
	in Section~\ref{sec2}.  Writing
	\[
	q_\ve
	:=
	f'(U)-\lambda-\sum_{j=1}^kV_j,
	\]
	Proposition~\ref{prop:error} gives
	\begin{equation*}\label{qweighted}
		\|q_\ve\phi\|_*
		=
		o(1)\|\phi\|_{**}.
	\end{equation*}
	
	Define
	\begin{equation}\label{Ldef}
		L\phi
		:=
		\Delta\phi
		+\left(\lambda+\sum_{j=1}^kV_j\right)\phi
	\end{equation}
	and
	\begin{equation*}\label{ne}
		N(\phi)
		:=
		f(U+\phi)-f(U)-f'(U)\phi+q_\ve\phi.
	\end{equation*}
	Then the equation for $\phi$ is
	\begin{equation*}\label{lv1}
		\begin{cases}
			L\phi=-E-N(\phi) & \text{in }\Omega,\\
			\phi=0 & \text{on }\partial\Omega.
		\end{cases}
	\end{equation*}
	
	After rescaling around the $j$-th core, $L$ tends to
	\[
	\mathcal L_0\psi:=\Delta\psi+e^w\psi.
	\]
	The bounded kernel of $\mathcal L_0$ is generated by the dilation mode and
	the two translation modes \cite{bp}.  The translation modes are odd with
	respect to the coordinate reflections, so only the dilation mode remains
	in the class \eqref{simmetria}.  We write
	\begin{equation*}\label{glizeta}
		Z_j(x)
		:=
		\frac{|x|^2-d_j^2}{|x|^2+d_j^2}
		=
		Z(x/d_j),
		\qquad
		Z(y):=\frac{|y|^2-1}{|y|^2+1}.
	\end{equation*}
	
	Fix $R_0>1$ and choose a smooth radial function
	$\zeta:[0,\infty)\to[0,1]$ satisfying
	\[
	\zeta=1
	\quad\text{on }[R_0^{-1},R_0],
	\qquad
	\zeta=0
	\quad\text{outside }[(2R_0)^{-1},2R_0].
	\]
	Set
	\begin{equation*}\label{etaj}
		\zeta_j(x):=\zeta(|x|/d_j).
	\end{equation*}
	We use the constant
	\begin{equation}\label{kappazeta}
		\kappa_\zeta
		:=
		\int_{\mathbb R^2}e^wZ^2\zeta\,dy>0,
		\qquad
		C_\zeta
		:=
		\frac{8\pi}{p\kappa_\zeta}.
	\end{equation}
	The quantity $C_\zeta$ depends on $\ve$ only through $p=2+\ve$; this
	harmless dependence is suppressed from the notation.
	
	For a symmetric $h\in L^\infty(\Omega)$, consider the problem
	\begin{equation}\label{l01}
		L\phi
		=
		h+\sum_{j=1}^k b_jc_jV_jZ_j\zeta_j
		\quad\text{in}\quad\Omega,
	\end{equation}
	\begin{equation}\label{l02}
		\phi=0\quad\text{on}\quad\partial\Omega,
	\end{equation}
	\begin{equation}\label{l3}
		\int_\Omega V_jZ_j\zeta_j\phi=0,
		\qquad j=1,\ldots,k.
	\end{equation}
	
	For fixed $(\ve,m)$, the weight $\beta_\ve$ is bounded away from zero,
	so $\|\cdot\|_{**}$ is an equivalent weighted norm on
	$C(\overline\Omega)$. We write $\mathcal C_{**}$ for
	$C(\overline\Omega)$ endowed with this norm.
	
	\begin{proposition}\label{p1}
		There exist $\ve_0>0$ and $C>0$, depending only on the fixed parameter
		box, such that for every $0<\ve<\ve_0$ and every symmetric
		$h\in L^\infty(\Omega)$, problem
		\eqref{l01}--\eqref{l3} has a unique solution
		$(\phi,c_1,\ldots,c_k)$.  Moreover
		\begin{equation}\label{est}
			\|\phi\|_{**}
			+
			\max_{1\le j\le k}|c_j|
			\le
			C\|h\|_*.
		\end{equation}
		For fixed $\ve$, the solution depends continuously on $(m,h)$ in the
		corresponding weighted norms.
	\end{proposition}
	
	The proof is given in Section~\ref{seclinear}.  We denote by
	$T_{\ve,m}h$ the $\phi$-component of the solution of the projected linear
	problem \eqref{l01}--\eqref{l3}.
	
	Consider the nonlinear projected problem
	\begin{align}
		L\phi
		&=
		-E-N(\phi)+\sum_{j=1}^k b_jc_jV_jZ_j\zeta_j\quad\text{in}\quad \Omega,
		\label{nl1}\\
		\phi&=0\quad\text{on}\quad \partial\Omega,
		\label{nl2}\\
		\int_\Omega V_jZ_j\zeta_j\phi&=0,
		\qquad j=1,\ldots,k.
		\label{nl3}
	\end{align}
	
	By Proposition~\ref{p1}, this is equivalent to
	\begin{equation*}\label{fp}
		\phi
		=
		T_{\ve,m}(-E-N(\phi))
		=:A(\phi).
	\end{equation*}
	For fixed $M>0$, let
	\[
	\mathcal B_M
	:=
	\{\phi\in\mathcal C_{**}:\|\phi\|_{**}\le M\ve\}.
	\]
	
	\begin{lemma}[nonlinear remainder]\label{lem:nonlinear}
		For every fixed $M>0$ there is a constant $C_M$, independent of small
		$\ve$, such that
		\begin{align}
			\|f(U+\phi)-f(U)-f'(U)\phi\|_*
			&\le C_M\|\phi\|_{**}^2,
			\label{quadN}\\
			\|N(\phi)\|_*
			&\le C_M\|\phi\|_{**}^2+o(1)\|\phi\|_{**},
			\label{Nbound}
		\end{align}
		for $\phi\in{\mathcal B}_M$.  Moreover
		\begin{equation}\label{Nlipschitz}
			\|N(\phi_1)-N(\phi_2)\|_*
			\le
			\left[
			C_M(\|\phi_1\|_{**}+\|\phi_2\|_{**})+o(1)
			\right]
			\|\phi_1-\phi_2\|_{**}.
		\end{equation}
	\end{lemma}
	
	\begin{proof}
		By Taylor's formula it is enough to prove
		\begin{equation}\label{fsecond-weight}
			\sup_{|\theta|\le1}
			|f''(U+\theta\phi)|\,\beta_\ve^2
			\le C_M\varrho_\ve ,
		\end{equation}
		where
		\[
		f''(s)= \lambda p(1+p+ps^p)s^{p-1}e^{s^p}, \quad s>0.
		\]
		On the $j$-th core,
		\[
		f''(U)
		=
		O(\gamma_j^{p-1}V_j)
		=
		O(b_j^{-1}V_j),
		\]
		whereas $\beta_\ve\asymp b_j$.  Hence
		\[
		f''(U)\beta_\ve^2
		\le Cb_jV_j
		\le C\varrho_\ve.
		\]
		Since $\|\phi\|_{**}\le M\ve$, the same estimates hold uniformly
		for $U+\theta\phi$, $|\theta|\le1$. On an inter-bubble transition
		region, \eqref{betaUbound} gives
		\[
		|\phi|U^{p-1}\le M\ve\,\beta_\ve U^{p-1}\le C_M,
		\]
		so the mean value theorem yields
		\[
		|(U+\theta\phi)_+^p-U^p|\le C_M.
		\]
		Thus the relevant exponential factors are comparable up to a constant
		which depends only on $M$. On the bounded-amplitude part of the outer
		bridge one has $\beta_\ve=b_1\to0$, hence $\phi=o(1)$ uniformly and
		the same conclusion follows directly.
		
		In a high-amplitude transition region,
		\[
		f''(U)
		\le
		C U^{2p-2}\lambda Ue^{U^p}.
		\]
		By \eqref{betaUbound},
		\[
		\beta_\ve^2U^{2p-2}
		\le\frac{C}{\ve^2},
		\]
		and therefore
		\[
		\chi(U)f''(U)\beta_\ve^2
		\le
		C\ve^{-2}{\mathcal D}^{\rm hi}_\ve
		\le C\varrho_\ve.
		\]
		On the low-amplitude part, which lies in the outer bridge,
		$0\le U\le2$, $\beta_\ve=b_1$, and
		\[
		f''(U)\le C U^{p-1}.
		\]
		Hence
		\[
		[1-\chi(U)]f''(U)\beta_\ve^2
		\le Cb_1^2
		\le Cb_1
		\le C\varrho_\ve.
		\]
		The exterior region is even easier because $f''(U)$ is uniformly bounded
		and $\beta_\ve=b_1\to0$.  This proves \eqref{fsecond-weight} and hence
		\eqref{quadN}.  Estimate \eqref{Nbound} follows from
		Proposition \ref{prop:error}, and \eqref{Nlipschitz} follows by the same
		Taylor argument.
	\end{proof}
	
	By Proposition \ref{p1}, Proposition \ref{prop:error}, and
	Lemma \ref{lem:nonlinear},
	\[
	\|A(\phi)\|_{**}
	\le
	C\ve+C_M\|\phi\|_{**}^2+o(1)\|\phi\|_{**}.
	\]
	Choosing $M$ large and then $\ve$ small, $A$ maps ${\mathcal B}_M$ into
	itself.  Moreover \eqref{Nlipschitz} shows that its Lipschitz constant on
	${\mathcal B}_M$ is
	\[
	O(\ve)+o(1)=o(1).
	\]
	Thus the contraction mapping theorem applies.
	
	\begin{proposition}\label{ls}
		For all sufficiently small $\ve$, problem
		\eqref{nl1}--\eqref{nl3} has a unique solution
		\[
		\phi=\phi(m)\in\mathcal C_{**}
		\]
		in the ball $\mathcal B_M$, for $M$ fixed sufficiently large.  It
		satisfies
		\begin{equation*}\label{phi-estimate}
			\|\phi(m)\|_{**}\le C\ve
		\end{equation*}
		uniformly for $m$ in \eqref{emme}, and the maps
		$m\mapsto\phi(m)$ and $m\mapsto c_j(m)$ are continuous.
	\end{proposition}
	
	Evaluating the projected problem at this solution, we obtain continuous
	coefficient functions
	\[
	c_j=c_j(m_1,\ldots,m_k),\qquad j=1,\ldots,k.
	\]
	Equation \eqref{main} is solved precisely when
	\begin{equation*}\label{sys}
		c_j(m)=0,\qquad j=1,\ldots,k.
	\end{equation*}
	The finite-dimensional analysis will be carried out directly from these
	coefficients, without differentiating a reduced energy.
	
	Define
	\begin{equation*}\label{psi1der}
		\psi_\ve(m_1)
		:=
		\log\frac8{\lambda p m_1}-1+\frac{h_\lambda(0)}2.
	\end{equation*}
	Thus $\psi_\ve$ depends on $\ve$ only through $p=2+\ve$ and
	\begin{equation*}\label{psi0}
		\psi_\ve\longrightarrow\psi_0,
		\qquad
		\psi_0(m_1)
		=
		\log\frac4{\lambda m_1}-1+\frac{h_\lambda(0)}2
	\end{equation*}
	uniformly on compact subsets of $(0,\infty)$.
	
	\medskip
	The projection system is anisotropic: the relative-height equations appear
	at order $\ve$, whereas the remaining scalar equation appears only after a
	telescopic cancellation, at order $G_1^{-1}\asymp\ve/|\log\ve|$. We
	therefore solve the projection equations directly rather than differentiate
	a reduced energy.
	
	\begin{proposition}[fast and slow projection equations]\label{prop:triangular}
		Put $m_{k+1}:=0$.  Uniformly for
		$m\in[\delta,\delta^{-1}]^k$, the projection coefficients satisfy
		\begin{equation}\label{fastcj}
			c_j(m)
			=
			-C_\zeta\,\ve
			\left(\frac12-m_j+m_{j+1}\right)
			+o(\ve),
			\qquad j=2,\ldots,k,
		\end{equation}
		and
		\begin{equation}\label{slowcj}
			\sum_{j=1}^k c_j(m)
			=
			-\frac{C_\zeta}{G_1}\,
			\psi_\ve(m_1)
			+o(G_1^{-1}),
			\qquad G_1=\gamma_1^p.
		\end{equation}
		The remainders are uniform in the parameter box.
	\end{proposition}
	
	\begin{proof}[Proof of Theorem \ref{teo}]

		Equations \eqref{fastcj}--\eqref{slowcj} have
		$k-1$ fast directions and one slow direction.  The slow equation is not
		$c_1=0$ by itself: the order-$\ve$ parts cancel only after taking the
		combination $c_1+\cdots+c_k$.
		
		Proposition \ref{prop:triangular} is proved in Section
		\ref{sec:directprojection}.  Assuming it for the moment, define
		\[
		{\mathcal F}_\ve(m)
		:=
		\left(
		-\frac{G_1}{C_\zeta}\sum_{j=1}^k c_j,
		-\frac{1}{C_\zeta\ve}c_2,\ldots,
		-\frac{1}{C_\zeta\ve}c_k
		\right).
		\]
		Then, uniformly on compact subsets of the parameter box,
		\[
		{\mathcal F}_\ve(m)\longrightarrow
		{\mathcal F}_0(m)
		:=
		\left(
		\psi_0(m_1),
		\frac12-m_2+m_3,\ldots,
		\frac12-m_{k-1}+m_k,
		\frac12-m_k
		\right).
		\]
		When $k=1$, only the first component is present; the convention $m_{k+1}=0$ is then vacuous.
		
		The limiting system is triangular. If $k\ge2$, starting from the last equation gives
		\[
		m_k^0=\frac12,
		\qquad
		m_j^0=m_{j+1}^0+\frac12
		=\frac{k-j+1}{2},
		\quad j=2,\ldots,k-1.
		\]
		The first equation gives
		\[
		m_1^0=\frac4\lambda e^{h_\lambda(0)/2-1}.
		\]
		We take the fixed constant $\delta$ in \eqref{emme} sufficiently small so
		that every component of $m^0$ lies in the interior of
		$[\delta,\delta^{-1}]^k$. The Jacobian $D{\mathcal F}_0(m^0)$ is triangular with nonzero diagonal.
		Hence the Brouwer degree of ${\mathcal F}_0$ in a sufficiently small
		product neighborhood of $m^0$ is nonzero.  By uniform convergence, the same
		holds for ${\mathcal F}_\ve$ for all small $\ve$.  We obtain
		$m^\ve\to m^0$ such that
		\[
		c_2(m^\ve)=\cdots=c_k(m^\ve)=0,
		\qquad
		\sum_{j=1}^k c_j(m^\ve)=0.
		\]
		Therefore also $c_1(m^\ve)=0$, and
		\[
		u_\ve=U(m^\ve)+\phi(m^\ve)
		\]
		solves the extended equation
		\[
		-\Delta u_\ve=\lambda u_\ve e^{(u_\ve^+)^p},
		\qquad u_\ve=0\quad\hbox{on }\partial\Omega.
		\]
		Let $u_\ve^-:=\max\{-u_\ve,0\}$.  Testing by $u_\ve^-$ gives
		\[
		\int_\Omega|\nabla u_\ve^-|^2
		=
		\lambda\int_\Omega (u_\ve^-)^2.
		\]
		Since $\lambda<\lambda_1(\Omega)$, we conclude that $u_\ve^-=0$.
		The solution is nontrivial by the core expansion, and the strong maximum
		principle gives $u_\ve>0$ in $\Omega$.  Hence it solves exactly
		\eqref{main}.
		
		Moreover, \eqref{slowcj} implies
		\[
		\psi_\ve(m_1^\ve)=o(1),
		\]
		and therefore the more precise relation
		\begin{equation}\label{m1-refined}
			m_1^\ve
			=
			\frac8{\lambda p}e^{h_\lambda(0)/2-1}+o(1).
		\end{equation}
		
		Finally, let
		\[
		x_\ve:=\frac{k\ve G_1}{2}.
		\]
		Since $m_1=G_1e^{-x_\ve}$, we have
		\[
		x_\ve e^{-x_\ve}=\frac{k\ve m_1}{2}.
		\]
		By the large-branch estimate following \eqref{def1},
		\[
		x_\ve
		=
		|\log\ve|+\log|\log\ve|+O(1).
		\]
		Consequently
		\[
		G_1
		=
		\frac{2}{k\ve}|\log\ve|\,(1+o(1)),
		\]
		and therefore
		\[
		\gamma_1
		=
		\sqrt{\frac2k}\,
		|\log\ve|^{1/2}\ve^{-1/2}(1+o(1)).
		\]
		For $j\ge2$, using
		\[
		\gamma_j
		=
		\frac{\gamma_1}{\prod_{i=2}^jm_i}\,\ve^{-(j-1)}
		\]
		and $m_i\to m_i^0=(k-i+1)/2$, we obtain
		\[
		\gamma_j
		=
		\alpha_j|\log\ve|^{1/2}\ve^{1/2-j}(1+o(1)),
		\qquad
		\alpha_j
		=
		\frac{2^{j-1}(k-j)!}{(k-1)!}\sqrt{\frac2k}.
		\]
		Finally set
		\[
		\widehat\gamma_j:=a_j\gamma_j.
		\]
		The matching lemma gives $a_j=1+O(\ve)$, hence
		$\widehat\gamma_j/\gamma_j\to1$.  For fixed $y$ away from the origin,
		\eqref{local-U-exact} and Proposition \ref{ls} yield
		\begin{align*}
			p\gamma_j^{p-1}
			\bigl[u_\ve(d_jy)-\widehat\gamma_j\bigr]
			&=
			w(y)-4\Gamma_j\log|y|
			+pG_jr_{j,\ve}(y)
			+p\gamma_j^{p-1}\phi(d_jy)\\
			&=w(y)+o(1),
		\end{align*}
		uniformly on compact subsets of $\mathbb R^2\setminus\{0\}$, since
		$\Gamma_j=O(\ve)$, $G_jr_{j,\ve}=o(1)$ and
		$\|\phi\|_{**}=O(\ve)$.  This proves the profile statement and completes
		Theorem \ref{teo}.

	\end{proof}
	
	\medskip
	\medskip
	The remainder of the paper is devoted to the linear theory and to the direct expansion of the projection equations used above.

	\section{The projected linear problem}\label{seclinear}
	
	We prove Proposition \ref{p1}.  Near each concentration scale we use
	nondegeneracy of the Liouville bubble.  The estimates are then propagated
	through the regions between consecutive scales by the bridge barriers
	constructed in Section~\ref{sec2}.
	
	\begin{lemma}[a priori estimate]\label{lemma1}
		There exist $C>0$ and $\ve_0>0$ such that every solution of
		\eqref{l01}--\eqref{l3} satisfies
		\begin{equation}\label{estt}
			\|\phi\|_{**}+\max_{1\le j\le k}|c_j|
			\le C\|h\|_*
		\end{equation}
		for $0<\ve<\ve_0$.
	\end{lemma}
	
	\begin{proof}
		We argue by contradiction.  If \eqref{estt} fails, after division by
		\[
		M_n
		:=
		\|\phi_n\|_{**}
		+\max_{1\le j\le k}|c_{jn}|
		\]
		we may assume
		\begin{equation}\label{linear-normalization}
			\|\phi_n\|_{**}
			+\max_j|c_{jn}|
			=1,
			\qquad
			\|h_n\|_*\longrightarrow0.
		\end{equation}
		For readability we suppress the subscript $n$ below.
		
		The contradiction argument has four parts. First we show that the
		projection coefficients remain bounded. We then rescale around each core;
		the limiting equation still contains the possible limit of the
		corresponding coefficient $c_i$. A Wronskian identity for the radial
		average forces that limit to vanish, and the Liouville nondegeneracy then
		eliminates the core limit of $\phi$. Away from the origin the limiting
		operator is $\Delta+\lambda$, so the assumption
		$\lambda<\lambda_1(\Omega)$ eliminates the exterior limit. Finally the
		bridge supersolutions propagate the vanishing from the cores and the
		exterior through the whole domain.
		
		\smallskip
		\noindent
		\emph{Step 1: uniform boundedness of the projection coefficients.}
		Choose a fixed radial cutoff $\widehat\zeta$ such that
		\[
		\widehat\zeta=1\quad\hbox{on }\supp\zeta,
		\qquad
		\supp\widehat\zeta
		\subset\{(4R_0)^{-1}<|y|<4R_0\},
		\]
		and put
		\[
		\psi_i(x):=Z_i(x)\widehat\zeta(x/d_i).
		\]
		Testing \eqref{l01} against $\psi_i$ and dividing by $b_i$ gives
		\begin{equation}\label{matrix-c}
			\sum_{j=1}^kK_{ij}^\ve c_j=r_i^\ve,
		\end{equation}
		where
		\[
		K_{ij}^\ve
		=
		\frac{b_j}{b_i}
		\int_\Omega V_jZ_j\zeta_j\psi_i.
		\]
		For all sufficiently small $\ve$, the supports at different scales are
		disjoint, so
		\[
		K_{ij}^\ve=0,\qquad i\ne j,
		\]
		whereas
		\[
		K_{ii}^\ve
		=
		\int_{\mathbb R^2}e^wZ^2\zeta
		=
		\kappa_\zeta>0.
		\]
		On the fixed rescaled support of $\psi_i$, the normalization
		\eqref{linear-normalization} gives $|\phi|\le Cb_i$.  Since the
		derivatives of $\psi_i$ have the natural sizes $d_i^{-1}$ and $d_i^{-2}$,
		self-adjointness and \eqref{superalgebraic-separation} yield
		\[
		\frac1{b_i}\left|\int_\Omega L\phi\,\psi_i\right|\le C,
		\qquad
		\frac1{b_i}\left|\int_\Omega h\,\psi_i\right|
		\le C\|h\|_*=o(1).
		\]
		Thus \eqref{matrix-c} implies
		\begin{equation}\label{cibounded}
			|c_i|\le C,\qquad i=1,\ldots,k.
		\end{equation}
		
		\smallskip
		\noindent
		\emph{Step 2: simultaneous vanishing of the coefficients and of the core
			limits.}
		Fix $i$.  Passing to a subsequence, assume
		\[
		c_i\longrightarrow\bar c_i.
		\]
		Set
		\[
		\widehat\phi_i(y)
		:=
		b_i^{-1}\phi(d_iy).
		\]
		The norm \eqref{starnorm2} makes $\widehat\phi_i$ uniformly bounded on
		compact subsets of $\mathbb R^2$. The smaller bubbles do not leave a
		singular source at $y=0$: if $\ell>i$, then
		\[
		\int_{|y|\le R} d_i^2V_\ell(d_i y)
		|\widehat\phi_i(y)|\,dy
		\le C\frac{b_\ell}{b_i}\longrightarrow0.
		\]
		Thus the apparent puncture is removable in the limiting equation.
		Proposition \ref{prop:error}, \eqref{cibounded}, and the scale separation
		then imply that a subsequence converges locally to a bounded symmetric
		solution of
		\begin{equation}\label{limit-projected}
			\Delta\widehat\phi+e^w\widehat\phi
			=
			\bar c_i\,e^wZ\zeta
			\quad\hbox{in }\mathbb R^2.
		\end{equation}
		
		We next prove that $\bar c_i=0$. This conclusion cannot be obtained by
		applying nondegeneracy immediately, because the limiting equation is still
		forced in the dilation direction. The radial average separates this forcing
		from the non-radial modes. Let
		\[
		\overline\phi(r)
		:=
		\frac1{2\pi}\int_0^{2\pi}
		\widehat\phi(r,\theta)\,d\theta.
		\]
		Since the right hand side of \eqref{limit-projected} is radial,
		$\overline\phi$ satisfies
		\[
		\overline\phi''+\frac1r\overline\phi'
		+e^w\overline\phi
		=
		\bar c_i e^wZ\zeta.
		\]
		The radial kernel function
		\[
		Z(r)=\frac{r^2-1}{r^2+1}
		\]
		satisfies the corresponding homogeneous equation.  Hence the Wronskian
		identity is
		\begin{equation}\label{Wronskian}
			\left[
			r\left(
			Z\overline\phi'-Z'\overline\phi
			\right)
			\right]'
			=
			\bar c_i\,r e^wZ^2\zeta.
		\end{equation}
		Near $r=0$ both functions are bounded and the left hand side has zero
		boundary Wronskian.  Outside the support of $\zeta$,
		$\overline\phi$ solves the homogeneous radial equation.  Its second
		independent radial solution grows like $\log r$ at infinity; boundedness
		therefore forces $\overline\phi$ to be a multiple of $Z$ for large $r$,
		and the Wronskian again vanishes.  Integrating
		\eqref{Wronskian} from $0$ to $\infty$ gives
		\[
		\bar c_i
		\int_0^\infty r e^wZ^2\zeta\,dr
		=0.
		\]
		The integral is strictly positive, hence
		\begin{equation*}\label{ci-limit-zero}
			\bar c_i=0.
		\end{equation*}
		
		Equation \eqref{limit-projected} is now homogeneous.  By the
		nondegeneracy of the Liouville bubble, every bounded solution is a linear
		combination of the dilation and translation modes.  The two coordinate
		reflections eliminate the translation modes.  Finally, the orthogonality
		condition \eqref{l3} passes to the limit and gives
		\[
		\int_{\mathbb R^2}e^wZ\zeta\,\widehat\phi=0,
		\]
		so the dilation coefficient also vanishes.  Thus
		$\widehat\phi\equiv0$.
		
		Since every convergent subsequence of $c_i$ has zero limit,
		\begin{equation}\label{cjzero}
			c_i\longrightarrow0,
			\qquad i=1,\ldots,k.
		\end{equation}
		Likewise, for every fixed $R>1$,
		\begin{equation}\label{corezero}
			\max_{j<k}
			b_j^{-1}\|\phi\|_{L^\infty(d_j/R<|x|<Rd_j)}
			+
			b_k^{-1}\|\phi\|_{L^\infty(|x|<Rd_k)}
			\longrightarrow0.
		\end{equation}
		
		\smallskip
		\noindent
		\emph{Step 3: the exterior region.}
		On compact subsets of $\overline\Omega\setminus\{0\}$ the bubble
		potentials tend to zero.  Therefore a subsequence of $b_1^{-1}\phi$
		converges locally to a bounded solution of
		\[
		\Delta\phi+\lambda\phi=0
		\quad\hbox{in }\Omega\setminus\{0\},
		\qquad
		\phi=0\quad\hbox{on }\partial\Omega.
		\]
		The singularity at the origin is removable.  Since
		$0<\lambda<\lambda_1(\Omega)$, the only solution is zero.  Thus
		\begin{equation}\label{exteriorzero}
			b_1^{-1}\|\phi\|_{L^\infty(\Omega\setminus B(0,r_0))}
			\longrightarrow0.
		\end{equation}
		
		It remains to propagate these estimates through the long annuli between
		consecutive scales.
	\end{proof}
	
	\begin{lemma}[propagation across the logarithmic bridges]
		\label{lem:two-scale-barrier}
		Fix $R>1$ sufficiently large and, if necessary, decrease the fixed radius $r_0>0$.  There exists $C$, independent of $\ve$, such that every solution
		of \eqref{l01}--\eqref{l3} satisfies
		\begin{equation}\label{bridgeestimate}
			\|\phi\|_{**}
			\le
			C\left[
			\max_{j<k} b_j^{-1}
			\|\phi\|_{L^\infty(d_j/R<|x|<Rd_j)}
			+b_k^{-1}\|\phi\|_{L^\infty(|x|<Rd_k)}
			+b_1^{-1}\|\phi\|_{L^\infty(\Omega\setminus B(0,r_0))}
			+\|h\|_*
			\right].
		\end{equation}
	\end{lemma}
	
	\begin{proof}
		Choose $R$ larger than the support radius of all the cutoff functions
		$\zeta_j$.  Hence the projection terms in \eqref{l01} vanish on every
		bridge
		\[
		Rd_j<|x|<d_{j-1}/R,
		\qquad j=2,\ldots,k.
		\]
		Let ${\mathcal S}_j=\beta_\ve+P_j$ be the supersolution given by
		Lemma \ref{lem:affine-bridge}.  It satisfies
		\[
		\beta_\ve\le{\mathcal S}_j\le\frac32\beta_\ve,
		\qquad
		L{\mathcal S}_j\le-\varrho_\ve.
		\]
		Since ${\mathcal S}_j>0$ and $L{\mathcal S}_j<0$, the operator $L$
		satisfies the maximum principle on the bridge. Indeed, if $v/{\mathcal S}_j$ had a negative interior minimum at $x_0$,
		then $\nabla(v/{\mathcal S}_j)(x_0)=0$ and
		$\Delta(v/{\mathcal S}_j)(x_0)\ge0$; expanding $L v$ at $x_0$ and using
		$L{\mathcal S}_j<0$ gives $L v(x_0)>0$, a contradiction. Thus the
		comparison principle
		is available below even though the zeroth-order coefficient of $L$ is
		positive.
		If
		\[
		{\mathcal M}_j
		=
		\max\left\{
		\frac{\|\phi\|_{L^\infty(|x|=d_{j-1}/R)}}
		{\beta_\ve(A_j)},
		\frac{\|\phi\|_{L^\infty(|x|=Rd_j)}}
		{\beta_\ve(B_j)}
		\right\},
		\]
		then, since $|h|\le\|h\|_*\varrho_\ve$,
		\[
		L\left[
		({\mathcal M}_j+\|h\|_*){\mathcal S}_j-\phi
		\right]
		\le0,
		\]
		and the same inequality holds with $\phi$ replaced by $-\phi$.
		The boundary values are non-negative.  The maximum principle therefore
		gives
		\[
		|\phi(x)|
		\le
		C({\mathcal M}_j+\|h\|_*)\beta_\ve(x)
		\]
		throughout the bridge.  Since
		\[
		\beta_\ve(A_j)\asymp b_{j-1},
		\qquad
		\beta_\ve(B_j)\asymp b_j,
		\]
		the boundary quotients are controlled by the core terms on the right hand
		side of \eqref{bridgeestimate}.
		
		The annulus
		\[
		Rd_1<|x|<r_0
		\]
		is treated in exactly the same way with the outer version of
		Lemma \ref{lem:bridge-moments}; here $\beta_\ve\equiv b_1$.
		No additional barrier is needed on $\Omega\setminus B(0,r_0)$.  By the
		definition of the weight, $\beta_\ve=b_1$ there, and hence
		\[
		\sup_{\Omega\setminus B(0,r_0)}
		\frac{|\phi|}{\beta_\ve}
		=
		b_1^{-1}\|\phi\|_{L^\infty(\Omega\setminus B(0,r_0))}.
		\]
		This is already one of the terms on the right-hand side of
		\eqref{bridgeestimate}.  Combining this observation with the estimates on
		the logarithmic bridges proves \eqref{bridgeestimate}.
	\end{proof}
	
	We can now complete the proof of Lemma \ref{lemma1}.  Applying
	Lemma \ref{lem:two-scale-barrier} to the contradiction sequence and using
	\eqref{corezero}, \eqref{exteriorzero}, and $\|h_n\|_*\to0$, we obtain
	$\|\phi_n\|_{**}\to0$.  Together with \eqref{cjzero}, this contradicts
	the combined normalization \eqref{linear-normalization}.  Thus
	\eqref{estt} is proved.
	
	\medskip
	\begin{proof}[Proof of Proposition \ref{p1}]
		Let
		\[
		X_\ve=
		\left\{\phi\in H_0^1(\Omega):
		\int_\Omega V_jZ_j\zeta_j\phi=0,
		\ j=1,\ldots,k\right\}.
		\]
		Testing \eqref{l01} against functions in $X_\ve$ eliminates the
		finite-dimensional projection terms. The resulting weak equation on
		$X_\ve$ is the identity plus a compact operator, hence Fredholm of index
		zero. Lemma \ref{lemma1} rules out a nontrivial homogeneous solution;
		Fredholm's alternative therefore gives existence and uniqueness. The
		coefficients $c_j$ are then recovered from the $k$ projection identities.  Estimate \eqref{estt} gives the uniform operator bound. We finally prove the asserted continuity in the parameters. Fix
		$\ve>0$ and let $m^{(n)}\to m$ in the parameter box.  For a fixed
		$h\in L^\infty(\Omega)$ with $\|h\|_*<\infty$, let $(\phi_n,c^{(n)})$ be the corresponding
		solutions.  Estimate \eqref{estt} gives uniform bounds.  Since, for fixed
		$\ve$, all scales $d_j(m)$ stay positive and depend continuously on $m$,
		elliptic compactness gives, after extraction,
		$\phi_n\to\phi$ locally in $C^1$ and uniformly in $\overline\Omega$.
		The coefficients $c_j^{(n)}$ converge after extraction as well.  Passing
		to the equation and to the orthogonality conditions shows that the limit
		solves the projected problem at $m$.  Uniqueness identifies the limit with
		$T_{\ve,m}h$.  Since every subsequence has the same limit, the whole
		sequence converges. For fixed $\ve$, the weights associated with parameters in
		the compact box are uniformly equivalent. The same compactness argument
		with $h_n\to h$ gives the asserted joint continuity in $(m,h)$. This proves
		Proposition \ref{p1}.
	\end{proof}
	
	\section{Direct expansion of the projection equations}
	\label{sec:directprojection}
	
	We derive the finite-dimensional equations directly from the projected
	PDE. The reduction has two different scales: $k-1$ equations are of order
	$\ve$, while one scalar equation is of order
	$G_1^{-1}\sim\ve/|\log\ve|$.
	
	It is useful to describe the expected outcome before carrying out the
	calculation. At the $j$-th core, $j\ge2$, the leading projection compares
	the interaction with the preceding scale and the contribution of the next
	inner scale. This gives, to first order,
	\[
	c_j\sim-C_\zeta\ve
	\left(\frac12-m_j+m_{j+1}\right),
	\qquad m_{k+1}=0.
	\]
	Starting with $j=k$, these equations determine successively
	$m_k,m_{k-1},\ldots,m_2$. The coefficient $c_1$ also has an order-$\ve$
	term, but at that order it contains $m_2$ rather than the remaining
	parameter $m_1$. The $m_1$ dependence appears only after the order-$\ve$
	terms are cancelled. The cancellation is obtained by summing all the
	projection equations. The resulting equation is of size $G_1^{-1}$ and
	selects $m_1$. The estimates below make this description uniform up to an
	error $o(G_1^{-1})$.
	
	Throughout this section we put
	\[
	Z(y)=\frac{|y|^2-1}{|y|^2+1},
	\qquad
	w(y)=\log\frac8{(1+|y|^2)^2}.
	\]
	The projection normalization is the constant $\kappa_\zeta$ defined in
	\eqref{kappazeta}. We also set $m_{k+1}=0$.
	
	\subsection{A local expansion with exact constant matching}
	
	For the projection computation we use a radius much larger than the
	logarithmic core radius used in the norm.  Set
	\[
	\mathfrak R_\ve:=\ve^{-2}.
	\]
	The scale separation is exponential in the quantities $G_j$, hence
	\[
	\mathfrak R_\ve d_j
	\ll
	\mathfrak R_\ve^{-1}d_{j-1},
	\qquad j=2,\ldots,k.
	\]
	Moreover
	\[
	\log\mathfrak R_\ve=2|\log\ve|
	=o(R_\ve),
	\]
	so this larger projection annulus is still contained in the region where
	the exact core expansion from Proposition \ref{prop:error} is valid.  We
	set
	\[
	{\mathcal A}_{j,\ve}
	=
	\{\mathfrak R_\ve^{-1}<|y|<\mathfrak R_\ve\}.
	\]
	
	\begin{lemma}[local form of the approximate solution]\label{lem:local-U-exact}
		Uniformly for $m\in[\delta,\delta^{-1}]^k$ and
		$y\in{\mathcal A}_{j,\ve}$,
		\begin{equation}\label{local-U-exact}
			\frac{U(d_jy)}{\gamma_j}
			=
			a_j+
			\frac{w(y)-4\Gamma_j\log|y|}{pG_j}
			+r_{j,\ve}(y),
		\end{equation}
		where
		\begin{equation}\label{local-rem}
			\sup_{1\le j\le k}\sup_{{\mathcal A}_{j,\ve}}
			G_j|r_{j,\ve}(y)|
			=o(G_1^{-1}).
		\end{equation}
		Moreover
		\begin{align}
			\Gamma_j
			&=
			\ve m_{j+1}+o(G_1^{-1}),
			\qquad j=1,\ldots,k,
			\label{Gamma-fine}\\
			-\frac{p\eta_j}{\gamma_j}
			&=
			2\ve m_j+o(G_1^{-1}),
			\qquad j=2,\ldots,k.
			\label{eta-fine}
		\end{align}
	\end{lemma}
	
	\begin{proof}
		Write
		\[
		u_i(x):=p\gamma_i^{p-1}U_i(x)
		=
		\log\frac1{(d_i^2+|x|^2)^2}-H_i^\lambda(x).
		\]
		Fix $j$ and put $x=d_jy$.  If $i<j$, then
		$d_j\mathfrak R_\ve/d_i\to0$ faster than every power of $\ve$, and
		\[
		u_i(d_jy)
		=
		-4\log d_i-h_\lambda(0)
		+
		O\!\left[
		\left(\frac{d_j\mathfrak R_\ve}{d_i}\right)^2
		+d_j\mathfrak R_\ve+d_i^{2-\nu}
		\right].
		\]
		If $i>j$, then
		$d_i\mathfrak R_\ve/d_j\to0$ faster than every power of $\ve$, and
		\[
		u_i(d_jy)
		=
		-4\log(d_j|y|)-h_\lambda(0)
		+
		O\!\left[
		\left(\frac{d_i\mathfrak R_\ve}{d_j}\right)^2
		+d_j\mathfrak R_\ve+d_i^{2-\nu}
		\right].
		\]
		The sign convention for $h_\lambda(0)$ is the one used in
		\eqref{defetaj} and \eqref{Bjexact}.
		
		After division by $\gamma_j$, the coefficients multiplying these
		remainders are at most algebraic powers of $\ve^{-1}$, while every scale
		ratio appearing above is exponentially small in one of the $G_i$'s.
		Consequently their total contribution satisfies
		\begin{equation*}\label{explicit-local-rem}
			G_j|r_{j,\ve}(y)|
			\le
			C\ve^{-C_k}
			\left[
			d_j\mathfrak R_\ve
			+
			\max_{i<j}
			\left(\frac{d_j\mathfrak R_\ve}{d_i}\right)^2
			+
			\max_{i>j}
			\left(\frac{d_i\mathfrak R_\ve}{d_j}\right)^2
			\right]
			=o(G_1^{-1}).
		\end{equation*}
		The constants from the outer bubbles $i<j$ cancel by the definition of
		$\eta_j$, while the inner bubbles $i>j$ give
		\[
		-4\Gamma_j\log|y|.
		\]
		The remaining constant is exactly ${\mathcal B}_j(L_j)$.  This proves
		\eqref{local-U-exact} and \eqref{local-rem}.
		
		We next expand the two finite sums which define $\Gamma_j$ and $\eta_j$.
		Since
		\[
		\frac{\gamma_j}{\gamma_{j+1}}=\ve m_{j+1},
		\qquad
		\frac{\gamma_j}{\gamma_{j+r}}
		=
		\ve^r m_{j+1}\cdots m_{j+r},
		\]
		we have
		\begin{align*}
			\Gamma_j
			&=
			(\ve m_{j+1})^{p-1}
			+
			O(\ve^{2(p-1)})\\
			&=
			\ve m_{j+1}
			+
			O(\ve^2|\log\ve|)
			=
			\ve m_{j+1}+o(G_1^{-1}),
		\end{align*}
		which is \eqref{Gamma-fine}.
		
		For $\eta_j$, formula \eqref{robust-d} gives
		\[
		-4\log d_{j-1}-h_\lambda(0)
		=
		2G_{j-1}\bigl(1+O(\ve)\bigr)
		+O(\log G_{j-1}).
		\]
		Hence the $i=j-1$ term equals
		\[
		-\frac2p\gamma_{j-1}\bigl(1+O(\ve)\bigr),
		\]
		whereas the sum over $i\le j-2$ is
		$O(\ve\gamma_{j-1})$.  Dividing by $\gamma_j$ and using
		$\gamma_{j-1}/\gamma_j=\ve m_j$ yields
		\[
		-\frac{p\eta_j}{\gamma_j}
		=
		2\ve m_j+O(\ve^2|\log\ve|)
		=
		2\ve m_j+o(G_1^{-1}).
		\]
		This proves \eqref{eta-fine}.
	\end{proof}
	
	Define
	\[
	z_j(y):=w(y)-4\Gamma_j\log|y|.
	\]
	The matching identity \eqref{matchingidentity} and
	\eqref{local-U-exact} yield
	\begin{align*}
		&\mu_j^2
		\exp\{U(d_jy)^p-G_j\}
		\nonumber\\
		&\qquad=
		\exp\left\{
		\sigma_jz_j(y)
		+
		\frac{p-1}{2pG_j}a_j^{p-2}z_j(y)^2
		+o(G_1^{-1})
		\right\}.
		\label{exp-exact-local}
	\end{align*}
	The expansion is uniform on ${\mathcal A}_{j,\ve}$, since
	$|z_j(y)|=O(\log\mathfrak R_\ve)=O(|\log\ve|)$ there.
	
	We introduce the ratio between the nonlinear term and the $j$-th
	Liouville potential,
	\begin{equation*}\label{Qjdef}
		Q_j(y)
		:=
		\frac{U(d_jy)}{\gamma_j}\,
		\mu_j^2
		\exp\{U(d_jy)^p-G_j-w(y)\}.
	\end{equation*}
	For $j\ge2$, the terms $G_j^{-1}$ are $o(G_1^{-1})$, and therefore
	\begin{equation}\label{Qj-fast}
		Q_j(y)-1
		=
		(a_j-1)
		+(\sigma_j-1)w(y)
		-4\sigma_j\Gamma_j\log|y|
		+o(G_1^{-1}).
	\end{equation}
	For $j=1$, the terms of order $G_1^{-1}$ must be retained:
	\begin{align}
		Q_1(y)-1
		={}&(a_1-1)
		+(\sigma_1-1)w(y)
		-4\sigma_1\Gamma_1\log|y|
		\nonumber\\
		&+\frac{w(y)}{pG_1}
		+\frac{p-1}{2pG_1}w(y)^2
		+o(G_1^{-1}).
		\label{Q1-slow}
	\end{align}
	All terms quadratic in $a_j-1$ or $\Gamma_j$ are
	$O(\ve^2|\log\ve|^2)=o(G_1^{-1})$.
	
	\subsection{The universal projection integrals}
	
	\begin{lemma}[Liouville projection integrals]
		The dilation mode $Z(y)=(|y|^2-1)/(|y|^2+1)$ satisfies
		\begin{align}
			\int_{\mathbb R^2}e^wZ&=0,
			\label{int0}\\
			\int_{\mathbb R^2}e^wZw&=-8\pi,
			\label{intw}\\
			\int_{\mathbb R^2}e^wZ\log|y|&=4\pi,
			\label{intlog}\\
			\int_{\mathbb R^2}e^wZw^2&=48\pi(1-\log2),
			\label{intw2}\\
			\int_{\mathbb R^2}e^wZ^2&=\frac{8\pi}{3}.
			\label{intz2}
		\end{align}
	\end{lemma}
	\begin{proof}
		With $t=r^2$, after integrating in the angular variable,
		\[
		2\pi r e^{w(r)}\,dr=\frac{8\pi}{(1+t)^2}\,dt,
		\qquad
		Z=\frac{t-1}{t+1},
		\qquad
		w=\log8-2\log(1+t).
		\]
		The first and last identities follow by direct integration. The remaining
		ones follow by differentiating
		\[
		\int_0^\infty \frac{t^{a-1}}{(1+t)^b}\,dt
		=B(a,b-a)
		\]
		with respect to $a$ and $b$ at the relevant integer values. This gives,
		in particular,
		\[
		\int_0^\infty\frac{t-1}{(1+t)^3}\log t\,dt=1,
		\qquad
		\int_{\mathbb R^2}e^wZw^2=48\pi(1-\log2),
		\]
		and hence \eqref{intw}--\eqref{intw2}.
	\end{proof}
	The truncation to
	$\mathfrak R_\ve^{-1}<|y|<\mathfrak R_\ve$ changes all the expressions
	appearing below by
	\[
	O\!\left(
	\mathfrak R_\ve^{-2}\log^2\mathfrak R_\ve
	\right)
	=
	O(\ve^4|\log\ve|^2)
	=
	o(G_1^{-1}).
	\]
	
	Set
	\begin{equation*}\label{Pjepsdef}
		\Pi_j^\ve
		:=
		\int_{{\mathcal A}_{j,\ve}}e^wZ\,[Q_j-1].
	\end{equation*}
	Only this local quantity will be used.  By the tail estimate above, the
	coefficients in its expansion may be evaluated with the universal
	integrals \eqref{int0}--\eqref{intz2} over $\mathbb R^2$.  Thus, using
	\eqref{Qj-fast} and \eqref{int0}--\eqref{intlog}, we obtain,
	for $j=2,\ldots,k$,
	\begin{equation}\label{Pjfast-sigma}
		\Pi_j^\ve
		=
		-8\pi\bigl[(\sigma_j-1)+2\sigma_j\Gamma_j\bigr]
		+o(G_1^{-1}).
	\end{equation}
	For $j=1$, \eqref{Q1-slow} and \eqref{intw2} give
	\begin{equation}\label{P1sigma}
		P_1^\ve
		=
		-8\pi\bigl[(\sigma_1-1)+2\sigma_1\Gamma_1\bigr]
		+
		\frac{4\pi}{G_1}(2-3\log2)
		+o(G_1^{-1}).
	\end{equation}
	The constant $2-3\log2$ is the universal
	second-order Liouville contribution which enters the slow equation.
	
	\subsection{Expansion of the matching levels}
	
	Put
	\begin{equation}\label{CjDj}
		C_j:=2\log(\lambda pG_j)-\log8-h_\lambda(0),
		\qquad
		D_j:=2\log(\lambda pG_j)-h_\lambda(0).
	\end{equation}
	Since $L_j/G_j=1-a_j^p$, the definition \eqref{Bjexact} can be rewritten
	exactly as
	\begin{equation}\label{aequation}
		p(a_j-1)
		=
		2(a_j^p-1)-\ve-\frac{p\eta_j}{\gamma_j}
		+\frac{C_j}{G_j}
		+\Gamma_j\left(2a_j^p+\frac{D_j}{G_j}\right).
	\end{equation}
	
	For $j\ge2$, $G_j^{-1}=o(G_1^{-1})$.  Combining
	\eqref{aequation}, \eqref{Gamma-fine}, and \eqref{eta-fine}, and using
	$a_j=1+O(\ve)$, gives
	\begin{equation*}\label{ajfast}
		a_j-1
		=
		\frac{\ve}{p}
		\bigl(1-2m_j-2m_{j+1}\bigr)
		+o(G_1^{-1}).
	\end{equation*}
	Hence
	\begin{equation}\label{Sjfast}
		(\sigma_j-1)+2\sigma_j\Gamma_j
		=
		\ve\left(\frac12-m_j+m_{j+1}\right)
		+o(G_1^{-1}).
	\end{equation}
	Here replacing $p$ by $2$ in the leading coefficient produces only an
	$O(\ve^2)$ error.
	
	For $j=1$, $\eta_1=0$ and the term $C_1/G_1$ has the same size as the
	slow scale.  Equation \eqref{aequation} yields
	\begin{equation*}\label{a1slow}
		a_1-1
		=
		\frac1p\left(
		\ve-2\Gamma_1-\frac{C_1}{G_1}
		\right)
		+o(G_1^{-1}),
	\end{equation*}
	and therefore
	\begin{equation}\label{S1slow}
		(\sigma_1-1)+2\sigma_1\Gamma_1
		=
		\ve\left(\frac12+m_2\right)
		-
		\frac{C_1}{2G_1}
		+o(G_1^{-1}).
	\end{equation}
	
	Combining \eqref{Pjfast-sigma} with \eqref{Sjfast}, we find
	\begin{equation}\label{Pjfast}
		\Pi_j^\ve
		=
		-8\pi\ve
		\left(\frac12-m_j+m_{j+1}\right)
		+o(G_1^{-1}),
		\qquad j=2,\ldots,k.
	\end{equation}
	Likewise, \eqref{P1sigma} and \eqref{S1slow} give
	\begin{equation}\label{P1slow}
		P_1^\ve
		=
		-8\pi\ve\left(\frac12+m_2\right)
		+
		\frac{4\pi}{G_1}
		\bigl[C_1+2-3\log2\bigr]
		+o(G_1^{-1}).
	\end{equation}
	
	\subsection{From the error projection to the coefficients $c_j$}
	
	Put
	\[
	{\mathcal T}_\ve:=\log\mathfrak R_\ve=2|\log\ve|,
	\qquad
	S_\ve:=\frac12|\log\ve|^2.
	\]
	We choose a radial cutoff $\chi_{j,\ve}(x)=\vartheta_\ve(t)$,
	$t=\log(|x|/d_j)$, such that
	\[
	\vartheta_\ve(t)=1\quad\hbox{for }|t|\le {\mathcal T}_\ve,
	\qquad
	\vartheta_\ve(t)=0\quad\hbox{for }|t|\ge {\mathcal T}_\ve+S_\ve,
	\]
	and
	\begin{equation}\label{slow-cutoff-der}
		|\vartheta_\ve'|\le\frac{C}{S_\ve},
		\qquad
		|\vartheta_\ve''|\le\frac{C}{S_\ve^2}.
	\end{equation}
	The support is still contained in the $j$-th core region because
	${\mathcal T}_\ve+S_\ve<2R_\ve$ for small $\ve$.  Testing the exact projected
	equation against $Z_j\chi_{j,\ve}$ gives
	\begin{equation}\label{test-cj}
		\kappa_\zeta c_j
		=
		\gamma_j^{p-1}
		\int_\Omega {\mathcal R}(m)Z_j\chi_{j,\ve}.
	\end{equation}
	Here
	\[
	{\mathcal R}(m)
	=
	\Delta(U+\phi)+f(U+\phi).
	\]
	On the support of $\chi_{j,\ve}$, the core expansion and
	$\|\phi\|_{**}=O(\ve)$ give $U+\phi>0$ for all small $\ve$; hence
	there $f(U+\phi)=\lambda(U+\phi)e^{(U+\phi)^p}$ and the local
	expansions below apply without change.
	For all small $\ve$, the support of $\zeta_i$ is disjoint from that of
	$\chi_{j,\ve}$ if $i\ne j$, while $\chi_{j,\ve}\equiv1$ on the support
	of $\zeta_j$.  Hence the off-diagonal terms vanish identically and
	\[
	\int_\Omega V_jZ_j^2\zeta_j\chi_{j,\ve}
	=
	\kappa_\zeta.
	\]
	
	We next compare ${\mathcal R}(m)$ with the error $E$.
	On the support of $\chi_{j,\ve}$ one has
	\[
	\beta_\ve\asymp b_j,
	\qquad
	{\mathcal D}^{\rm hi}_\ve
	=
	{\mathcal D}^{\rm lo}_\ve
	=0,
	\]
	and
	\[
	\gamma_j^{p-1}
	\int_{\supp\chi_{j,\ve}}\varrho_\ve
	\le C.
	\]
	Lemma \ref{lem:nonlinear} and $\|\phi\|_{**}\le C\ve$ therefore give
	\begin{equation*}\label{quadratic-projection}
		\gamma_j^{p-1}
		\int_\Omega
		\bigl|f(U+\phi)-f(U)-f'(U)\phi\bigr|
		|Z_j|\chi_{j,\ve}
		\le C\ve^2
		=o(G_1^{-1}).
	\end{equation*}
	The core expansion \eqref{core-fprime} gives, more quantitatively,
	\[
	|q_\ve|\beta_\ve
	\le
	C\left[
	\ve(1+|t|)
	+\frac{1+t^2}{G_j}
	\right]b_jV_j
	+\hbox{exponentially small terms},
	\]
	where $t=\log(|x|/d_j)$.  Since $|t|\le {\mathcal T}_\ve+S_\ve=O(|\log\ve|^2)$
	and the Liouville moments are finite,
	\begin{equation*}\label{linear-projection-small}
		\gamma_j^{p-1}
		\int_\Omega |q_\ve\phi Z_j|\chi_{j,\ve}
		\le
		C\ve^2|\log\ve|^2
		=
		o(G_1^{-1}).
	\end{equation*}
	Recalling the definition of $L$ in \eqref{Ldef}, by self-adjointness,
	\[
	\int L\phi\,Z_j\chi_{j,\ve}
	=
	\int\phi\,L(Z_j\chi_{j,\ve}).
	\]
	The identity
	\[
	LZ_j=
	\left(\sum_{i\ne j}V_i+\lambda\right)Z_j
	\]
	is used without discarding the interaction terms.  For a radial function
	of $t=\log(r/d_j)$,
	\[
	\Delta\chi_{j,\ve}
	=
	\frac1{r^2}\vartheta_\ve''(t).
	\]
	Since $|\phi|\le C\ve b_j$ throughout the support of the cutoff and
	$\gamma_j^{p-1}=b_j^{-1}$, \eqref{slow-cutoff-der} gives
	\[
	\gamma_j^{p-1}
	\int |\phi Z_j\Delta\chi_{j,\ve}|
	\le
	C\ve\int|\vartheta_\ve''(t)|\,dt
	\le
	C\frac{\ve}{S_\ve}.
	\]
	On the support of $\nabla\chi_{j,\ve}$ one has
	$|t|\ge {\mathcal T}_\ve$ and hence
	\[
	|\partial_tZ_j|\le Ce^{-2{\mathcal T}_\ve}=O(\ve^4).
	\]
	Therefore
	\[
	\gamma_j^{p-1}
	\int 2|\phi\nabla Z_j\cdot\nabla\chi_{j,\ve}|
	\le C\ve^5.
	\]
	Finally, the terms coming from
	$(\sum_{i\ne j}V_i+\lambda)Z_j\chi_{j,\ve}$ are exponentially small by
	the scale separation and by
	$d_j e^{{\mathcal T}_\ve+S_\ve}\to0$ faster than every power of $\ve$.  Consequently
	\begin{equation*}\label{cutoff-projection}
		\gamma_j^{p-1}
		\left|\int\phi\,L(Z_j\chi_{j,\ve})\right|
		\le
		C\frac{\ve}{|\log\ve|^2}+O(\ve^2)
		=
		o(G_1^{-1}).
	\end{equation*}
	
	Since $G_1^{-1}\asymp\ve/|\log\ve|$, we conclude that
	\begin{equation}\label{RminusE}
		\gamma_j^{p-1}
		\int_\Omega({\mathcal R}-E)Z_j\chi_{j,\ve}
		=o(G_1^{-1})
	\end{equation}
	uniformly for $j=1,\ldots,k$.
	
	\begin{lemma}[projection of the error]
		\label{lem:error-projection}
		Uniformly for $j=1,\ldots,k$,
		\begin{equation}\label{Eprojection-exact}
			\gamma_j^{p-1}
			\int_\Omega E\,Z_j\chi_{j,\ve}
			=
			\frac1p\Pi_j^\ve+o(G_1^{-1}).
		\end{equation}
	\end{lemma}
	
	\begin{proof}
		Set
		\[
		\Lambda_\ve:=e^{{\mathcal T}_\ve+S_\ve}.
		\]
		On the support of $\chi_{j,\ve}$,
		\[
		\frac{d_j}{\Lambda_\ve}
		\le |x|\le
		d_j\Lambda_\ve.
		\]
		From the definition of $E$,
		\begin{equation}\label{Edecomp-proj}
			E
			=
			\frac{b_j}{p}V_j(Q_j-1)
			-\frac1p\sum_{i\ne j}b_iV_i
			-\lambda U.
		\end{equation}
		
		We first estimate the cross-bubble potentials.  If $i<j$, then
		$|x|\ll d_i$ throughout the support of $\chi_{j,\ve}$ and
		\[
		V_i(x)\le\frac{C}{d_i^2}.
		\]
		Consequently
		\begin{equation*}\label{cross-outer}
			b_j^{-1}b_i
			\int_{\supp\chi_{j,\ve}}V_i
			\le
			C\frac{b_i}{b_j}
			\left(\frac{d_j\Lambda_\ve}{d_i}\right)^2
			=o(G_1^{-1}).
		\end{equation*}
		If $i>j$, then $d_i\ll |x|$ and
		\[
		V_i(x)\le C\frac{d_i^2}{|x|^4}.
		\]
		Hence
		\begin{equation*}\label{cross-inner}
			b_j^{-1}b_i
			\int_{\supp\chi_{j,\ve}}V_i
			\le
			C\frac{b_i}{b_j}
			\left(\frac{d_i\Lambda_\ve}{d_j}\right)^2
			=o(G_1^{-1}).
		\end{equation*}
		To justify the last relations, observe that every quotient $b_i/b_j$ is
		at most an algebraic power of $\ve^{-1}$, whereas
		\eqref{robust-d} gives exponential separation of the $d_i$'s.  The factor
		\[
		\log\Lambda_\ve
		=
		2|\log\ve|+\frac12|\log\ve|^2
		\]
		is negligible compared with the relevant $G_\ell$.
		
		The term $-\lambda U$ is even smaller.  The local expansion
		\eqref{local-U-exact} gives $|U|\le C\gamma_j$ on the support of
		$\chi_{j,\ve}$; therefore
		\begin{equation*}\label{lambdaU-proj}
			\gamma_j^{p-1}
			\int_{\supp\chi_{j,\ve}}\lambda|U|
			\le
			C G_jd_j^2\Lambda_\ve^2
			=o(G_1^{-1}).
		\end{equation*}
		
		It remains to compare the principal term with $\Pi_j^\ve$.  Since
		$\chi_{j,\ve}\equiv1$ for $|t|\le {\mathcal T}_\ve$, $t=\log(|x|/d_j)$,
		the difference is supported where
		\[
		{\mathcal T}_\ve\le |t|\le {\mathcal T}_\ve+S_\ve.
		\]
		The exact core expansion used in Proposition \ref{prop:error} yields there
		\[
		|Q_j-1|
		\le
		C\left[
		\ve(1+|t|)
		+\frac{1+t^2}{G_j}
		\right].
		\]
		In logarithmic coordinates,
		\[
		V_j(x)\,dx
		\le
		C e^{-2|t|}\,dt\,d\theta.
		\]
		Thus
		\begin{align}
			\int_{{\mathcal T}_\ve\le|t|\le {\mathcal T}_\ve+S_\ve}
			V_j|Q_j-1|
			&\le
			C e^{-2{\mathcal T}_\ve}
			\left[
			\ve(1+{\mathcal T}_\ve+S_\ve)
			+\frac{1+({\mathcal T}_\ve+S_\ve)^2}{G_j}
			\right]
			\nonumber\\
			&=
			o(G_1^{-1}),
			\label{Qtail-proj}
		\end{align}
		because $e^{-2{\mathcal T}_\ve}=\ve^4$.  Combining
		\eqref{Edecomp-proj}--\eqref{Qtail-proj} proves
		\eqref{Eprojection-exact}.
	\end{proof}
	
	Therefore \eqref{test-cj}, \eqref{RminusE}, and
	Lemma \ref{lem:error-projection} give
	\begin{equation}\label{cP}
		c_j
		=
		\frac{1}{p\kappa_\zeta}\Pi_j^\ve
		+o(G_1^{-1}).
	\end{equation}
	Using \eqref{Pjfast}, we obtain
	\begin{equation}\label{cj-fast-final}
		c_j
		=
		-C_\zeta\ve
		\left(\frac12-m_j+m_{j+1}\right)
		+o(G_1^{-1}),
		\qquad j=2,\ldots,k.
	\end{equation}
	In particular this proves \eqref{fastcj}, with a stronger remainder.
	
	It remains to identify the slow equation. The reason for considering the
	sum is visible from \eqref{cj-fast-final}: the fast equations form a
	discrete difference system. Their leading terms telescope, and the $m_2$
	term which occurs in the first projection is cancelled by the $m_2$ term
	coming from the sum of the remaining projections. This removes the entire
	order-$\ve$ contribution and exposes the next order, where $m_1$ enters.
	More precisely, summing \eqref{Pjfast} for $j=2,\ldots,k$ gives the
	telescopic identity
	\begin{equation*}\label{telescopicP}
		\sum_{j=2}^k\Pi_j^\ve
		=
		-8\pi\ve\left(\frac{k-1}{2}-m_2\right)
		+o(G_1^{-1}).
	\end{equation*}
	Combining with \eqref{P1slow},
	\begin{equation*}\label{Psum}
		\sum_{j=1}^k\Pi_j^\ve
		=
		\frac{4\pi}{G_1}
		\left[
		C_1+2-3\log2-k\ve G_1
		\right]
		+o(G_1^{-1}).
	\end{equation*}
	where $C_1$ is given in \eqref{CjDj}. Now the definition of $m_1$ gives the exact identity
	\[
	\frac{k\ve G_1}{2}
	=
	\log\frac{G_1}{m_1}.
	\]
	Consequently
	\begin{align*}
		C_1+2-3\log2-k\ve G_1
		&=
		2\log\frac{\lambda p m_1}{8}
		+2-h_\lambda(0)
		\nonumber\\
		&=
		-2\psi_\ve(m_1).
		\label{slow-bracket}
	\end{align*}
	Using \eqref{cP} and \eqref{kappazeta}, we arrive at
	\begin{equation*}\label{cj-slow-final}
		\sum_{j=1}^kc_j
		=
		-\frac{C_\zeta}{G_1}
		\psi_\ve(m_1)
		+o(G_1^{-1}).
	\end{equation*}
	This proves \eqref{slowcj} and completes the proof of Proposition
	\ref{prop:triangular}.
	
	The two levels of the reduction play different roles. The fast equations
	determine the relative geometry of the tower, while the slow equation fixes
	its outermost height through the domain-dependent quantity $h_\lambda(0)$.
	In particular, the Robin function enters only in the last scalar equation
	and not in the recursive determination of $m_2,\ldots,m_k$.
	
	\medskip
	\noindent
	Notice that the order-$\ve$ dependence on $m_2$ is still present in
	$c_1$.  It disappears only after summing the $k$ projection equations,
	which is why the equation for $m_1$ occurs at the smaller scale
	$G_1^{-1}$.
	
	\medskip
	\noindent\textbf{Acknowledgments.} M. del Pino has been supported by the Royal Society Research Professorship grant RP-R1-180114 and by the ERC/UKRI Horizon Europe grant ASYMEVOL, EP/Z000394/1. 

\end{document}